\documentclass{amsart}
\usepackage{amssymb}

\usepackage{graphicx}
\usepackage{color}

\usepackage{enumerate}
\usepackage[all]{xy}

\newtheorem{theorem}{Theorem}[section]
\newtheorem{lemma}[theorem]{Lemma}
\newtheorem{prop}[theorem]{Proposition}
\newtheorem{corollary}[theorem]{Corollary}
\theoremstyle{definition}
\newtheorem{definition}[theorem]{Definition}
\newtheorem{example}[theorem]{Example}

\theoremstyle{remark}
\newtheorem{remark}[theorem]{Remark}

\numberwithin{equation}{section}

\newcommand{\Ho}{\mathrm{Ho}}
\newcommand{\Hom}{\mathrm{Hom}}
\newcommand{\End}{\mathrm{End}}
\newcommand{\Ker}{\mathrm{Ker}}

\newcommand{\Img}{\mathrm{Im}}

\newcommand{\Sh}{\mathit{Sh}}

\newcommand{\xra}[1]{\xrightarrow{#1}}
\newcommand{\lra}{\longrightarrow}

\newcommand{\wt}{\widetilde}

\newcommand{\ov}[1]{\overline{#1}}

\newcommand{\ch}[1]{\mathrm{C^*_{#1}}}
\newcommand{\chb}[2]{\mathrm{C^{#1}_#2}}

\newcommand{\Einf}{\mathbf{E}}
\newcommand{\TW}{\text{\tiny{$\mathrm{TW}$}}}
\newcommand{\Sch}[1]{\mathrm{Sch}_{#1}}
\newcommand{\Sm}[1]{\mathrm{Sm}_{#1}}
\newcommand{\Com}{\mathcal{C}om}

\newcommand{\FF}{\mathbb{F}}
\newcommand{\CC}{\mathbb{C}}
\newcommand{\HH}{\mathbb{H}}

\newcommand{\PP}{\mathbb{P}}
\newcommand{\QQ}{\mathbb{Q}}
\newcommand{\RR}{\mathbb{R}}

\newcommand{\ZZ}{\mathbb{Z}}

\newcommand{\Aa}{\mathcal{A}}
\newcommand{\Bb}{\mathcal{B}}
\newcommand{\Cc}{\mathcal{C}}
\newcommand{\Dd}{\mathcal{D}}
\newcommand{\Ee}{\mathcal{E}}
\newcommand{\Ff}{\mathcal{F}}
\newcommand{\Gg}{\mathcal{G}}

\newcommand{\Ii}{\mathcal{I}}

\newcommand{\Ll}{\mathcal{L}}

\newcommand{\Nn}{\mathcal{N}}
\newcommand{\Oo}{\mathcal{O}}
\newcommand{\Pp}{\mathcal{P}}
\newcommand{\Qq}{\mathcal{Q}}
\newcommand{\Ss}{\mathcal{S}}
\newcommand{\Vv}{\mathcal{V}}
\newcommand{\Ww}{\mathcal{W}}

\newcommand{\Zz}{\mathcal{Z}}

\newcommand{\kk}{\mathbf{K}}

\begin{document}

\title[Sheaves of $E$-infinity algebras and applications]{Sheaves of $E$-infinity algebras and applications to 
 algebraic varieties and singular spaces}

\author{David Chataur}
\address[D. Chataur]{Laboratoire Ami\'{e}nois de Math\'{e}matique Fondamentale et Appliqu\'{e}e\\  Universit\'{e} de Picardie Jules Verne\\ 
33, rue Saint-Leu 80039 Amiens Cedex 1, France}
\email{david.chataur@u-picardie.fr}
\thanks{D. Chataur would like to thank the CRM and IMUB for their hospitality.}

\author{Joana Cirici}
\address[J. Cirici]{
Departament de Matem\`{a}tiques i Inform\`{a}tica\\ 
Universitat de Barcelona (UB)\\
Gran Via de les Corts Catalanes 585\\
08007 Barcelona}
\email{jcirici@ub.edu}
\thanks{J. Cirici would like to acknowledge partial support from the project AEI/FEDER, UE (PID2020-117971GB-C22) and the Serra H\'{u}nter Program.}

\subjclass[2020]{18D50, 32S35, 55N33}

\begin{abstract}
We describe the $E$-infinity algebra structure on the complex of singular cochains of a topological space, in the context of sheaf theory. As a first application, for any algebraic variety we define a weight filtration compatible with its $E$-infinity structure. This naturally extends the theory of mixed Hodge structures in rational 
homotopy to $p$-adic homotopy theory.
The spectral sequence associated to the weight filtration gives a new family of algebraic invariants of the varieties for any coefficient ring, carrying Steenrod operations.
As a second application, we promote Deligne's intersection complex computing intersection cohomology,
to a sheaf carrying E-infinity structures.
This allows for a natural interpretation of the Steenrod operations defined on the
intersection cohomology of any topological pseudomanifold.
\end{abstract}

\maketitle

\section{Introduction}
The singular cochain complex $C^*(X,R)$ of every topological space $X$,
with coefficients in a commutative ring $R$, has a natural structure of an $E_\infty$-algebra.
This structure encodes the commutativity of the cup product at the cochain level up to higher coherent homotopies
and turns out to be extremely powerful in order to classify homotopy types: as
shown by Mandell in \cite{Mandell}, when $R=\FF_p$ and for sufficiently nice spaces, the $E_\infty$-structure
captures the $p$-adic homotopy theory of the space.
This result can be understood as an algebraization
of $p$-adic homotopy theory, analogous to Sullivan's approach to 
rational homotopy via commutative dg-algebras.
Moreover, finite type nilpotent spaces are weakly equivalent if and only if
their singular cochains over $\ZZ$ are quasi-isomorphic as $E_\infty$-algebras
\cite{Mandell2}.

In this paper, we describe the $E_\infty$-algebra structure on $C^*(X,R)$ via sheaf theory.
The study of sheaves of $E_\infty$-algebras is not new and some related constructions appear
in \cite{HiSc}, \cite{ManCochains}, \cite{MaySheaf} and more recently in \cite{RR0}, \cite{RR} and \cite{Petersen}. 
In particular, Mandell used the
theory of homotopy limits of operadic algebras developed by Hinich and Schechtman
to define a cosimplicial normalization functor
in the category of $E_\infty$-algebras over $R$.
The same construction was used by May to outline 
a theory of $E_\infty$-algebras and \u{C}ech cochains as an approach to sheaf cohomology. More recently, Petersen has defined 
a global sections functor for sheaves of $E_\infty$-algebras using 
the theory of twisted algebras.
However, these studies are not sufficient for our purposes.
In this paper, we combine the cosimplicial normalization functor with 
Godement's cosimplicial resolution functor to 
define derived direct image $\RR_\Ee f_*$ and derived global sections $\RR_\Ee\Gamma(X,-)$ functors 
for sheaves of $E_\infty$-algebras 
on topological spaces.
This idea goes back to Godement, who already claimed in \cite{Godement} that 
the good behavior with respect to products, of his canonical cosimplical resolutions by flasque sheaves,
should allow one to study Steenrod operations in the context of sheaf theory.

We show that the assignment 
$X\mapsto \RR_\Ee\Gamma(X,\underline{R}_X)$, where $\underline{R}_X$ denotes the constant sheaf of $R$-modules on $X$,
defines a functor from the category of Hausdorff,
paracompact and locally contractible topological spaces, to the category of $E_\infty$-algebras, 
which is naturally quasi-isomorphic to the functor $C^*(-,R)$ defined by the complex of singular
cochains with its $E_\infty$-algebra structure (Theorem \ref{quistocochains_thm}).
In particular, the above functor is a cochain theory
in the sense of Mandell \cite{ManCochains}. Such a cochain theory is a
cochain-level refinement
of the Eilenberg-Steenrod axioms, defined as a contravariant functor from spaces to cochain complexes,
satisfying homotopy, excision, product and dimension axioms analogous to the usual axioms for cohomology.
If such a cochain theory lifts to the category of $E_\infty$-algebras,
then this lift is uniquely quasi-isomorphic to the 
usual functor of singular cochains.

When $\QQ=R$, the above recover results of Navarro-Aznar \cite{Na},
who used the Thom-Whitney simple functor to give a solution to the commutative 
cochains problem over the field of rational numbers via sheaf theory.

As already pointed out by May in \cite{MaySheaf}, the study of sheaves of $E_\infty$-algebras
has many potential applications to algebraic geometry. This is where the weight of this paper lies. We develop two main applications: to the weight filtration for algebraic varieties and to the intersection cohomology for topological pseudomanifolds.
Both of these theories owe much to Deligne and are based on sheaf-theoretic considerations.
\\

\textbf{The weight filtration}.
Deligne \cite{DeHII}, \cite{DeHIII} introduced a functorial increasing filtration $W$ on
the rational cohomology $H^*(X,\QQ)$ of every complex algebraic variety $X$, called the \textit{weight filtration}. 
This filtration relates the cohomology of the variety with cohomologies of smooth projective varieties.
The successive quotients
of this
filtration become pure Hodge structures of different weights, 
leading to the notion of \textit{mixed Hodge structure}.

The study of the rational homotopy of complex algebraic varieties using Deligne's mixed Hodge theory
and in particular, the weight filtration, has proven to be very fruitful.
The initial findings are due to Morgan \cite{Mo}, who
refined Deligne's construction to endow commutative dg-algebra models of smooth 
varieties with mixed Hodge structures. By introducing the Thom-Whitney simple functor and using the theory of cohomological descent, 
Navarro-Aznar \cite{Na} extended Morgan's theory to possibly singular varieties.
These constructions have led to numerous topological consequences, mostly related to the formality of algebraic varieties in the sense of rational homotopy.

The weight filtration has also been defined on the cohomology
with coefficients in an arbitrary commutative ring (see \cite{GilletSoule}, \cite{GN}). 
Also, Totaro observed in his ICM address \cite{To}, that the weight filtration
could also be defined on the cohomology with compact 
 supports of any complex or real analytic space endowed with an equivalence class of compactifications. 
 Totaro's remarks have been made precise in \cite{MCPI}, \cite{MCPII} 
 for real algebraic varieties and in \cite{CG3} for complex and real analytic spaces.
 Also, in \cite{LiPr}, cap and cup products are studied for the weight filtration on 
 compactly supported $\FF_2$-cohomology of real algebraic varieties.
 In view of these results, it is natural to ask for a weight filtration defined at the cochain level,
with coefficients in an arbitrary commutative ring, and compatible with the $E_\infty$-structure, thus generalizing all of the above constructions.

We obtain this weight filtration using the extension criterion of functors of \cite{GN}. This 
is based on the assumption that the target category
is a \textit{cohomological descent category} which, essentially, is a category $\Dd$ endowed with a saturated class of \textit{weak equivalences},
and a \textit{simple functor} sending every cubical codiagram of $\Dd$ to an object of $\Dd$ and
satisfying certain axioms analogous to those of the total complex of a double complex.
In Theorem \ref{descent1}, we use Mandell's cosimplicial normalization functor to obtain a cohomological descent structure on the category of $E_\infty$-algebras 
and prove a filtered version of this result in Theorem \ref{cohdesccats}.

The basic idea of the extension criterion of functors is that, given a functor compatible with smooth
blow-ups from the category of smooth schemes to a cohomological descent category, there
exists an extension to all schemes. Such an extension is essentially unique
and is compatible with general blow-ups. Here, we use a relative version of this result, which extends functors from the category of pairs $(X,U)$ where $j:U\hookrightarrow X$ is a smooth compactification of a smooth open variety $U$.
Given such a pair, we define the weight filtration
by taking the canonical filtration on $\RR_{\Ee}j_*\underline{R}_U$, thus promoting Deligne's approach to the multiplicative setting. After extension, this gives a functor
$\Ww_R$ from the category of complex schemes to a certain homotopy category of filtered $E_\infty$-algebras (Theorem \ref{mainteoweight}).
The \textit{weight spectral sequence} of $X$ is then defined as the 
spectral sequence associated to the filtered $E_\infty$-algebra $\Ww_R(X)$.

Over the rationals, the weight spectral sequence always degenerates at the $E_2$ page. 
However, this is not the case when working with arbitrary coefficients.
 Nevertheless, it converges to the cohomology of $X$ and from the second stage onwards, any page 
is a new and well-defined algebraic invariant of the variety, which carries a 
commutative bigraded structure as well as Steenrod and reduced $\ell$-th 
power operations when working over $\FF_2$ and $\FF_\ell$ with $\ell$ a prime odd number respectively.
In \cite{To}, Totaro speculates that the mixed
motive of a real analytic space should not involve much more information than 
the weight spectral sequence over $\FF_2$ with an action of a Steenrod algebra.
While the verification of this idea is beyond the scope of the present work,
we do extend our constructions to real and complex analytic spaces
showing that 
the $E_2$-term of the weight spectral sequence carries Steenrod operations
which are well-defined invariants for these spaces.
We develop some examples illustrating that this is a new non-trivial invariant.
\medskip

\textbf{Intersection cohomology}.
Intersection cohomology is a Poincar\'{e} duality cohomology theory for singular spaces,
introduced by Goresky and MacPherson.
It is defined for any topological pseudomanifold and
depends on the choice of a multi-index called \textit{perversity}, measuring how far cycles are
allowed to deviate from transversality. 
Goresky and MacPherson, first defined intersection homology in \cite{GMP1} as the homology
$IH^{\overline p}_*(X)$ of a subcomplex $IC_*^{\overline{p}}(X)$ of the ordinary chains of $X$,
given by those cycles which meet the singular locus of $X$ with a
controlled defect of transversality with respect to the chosen perversity and stratification.
Subsequently, Deligne proposed a sheaf-theoretic approach,
which was developed in \cite{GMP2}. In this case, 
intersection cohomology is defined as the hypercohomology
a complex of sheaves  $\Ii C_{\ov p}^*(X)$ defined by starting with the constant sheaf on the 
regular part of $X$ and iteratively extending it to bigger open sets determined by the stratification and truncating
it in the derived category with respect to the chosen perversity.
The intersection complex $\Ii C_{\ov p}^*(X)$ is uniquely characterized
by a set of axioms in the derived category of sheaves.

Intersection cohomology does not define an algebra, but it has a product compatible with perversities.
In order to promote Deligne's additive sheaf to a
sheaf carrying an $E_\infty$-algebra structure,
we adapt Hovey's \cite{Hov2} formalism of perverse algebras to
define the notion of \textit{perverse $E_\infty$-algebra}
as an $E_\infty$-algebra internal to the category of functors from the poset of perversities to cochain complexes.
We then construct a sheaf of perverse $E_\infty$-algebras whose underlying sheaf of complexes is Deligne's intersection complex and show, in Theorem \ref{uniquenessIC}, that this is uniquely characterized
by a set of axioms in the homotopy category of perverse $E_\infty$-algebras.
There exist different constructions related to the presence of
multiplicative structures in intersection cohomology, starting by Goresky and MacPherson's 
study of cup products in intersection cohomology. There is also an ad. hoc. construction 
of Steenrod operations in intersection cohomology due to Goresky \cite{Goresky}.
Also, recent work of the first
author together with Saralegi and Tanr\'{e}
studies questions related to multiplicative aspects of intersection cohomology, both over $\QQ$ and over $\FF_2$ (see \cite{CSTGor}, \cite{CSTBl},\cite{CST}, \cite{CSTSheaf}).
Our uniqueness result ensures that the intersection complex of perverse $E_\infty$-algebras defined in this paper, recovers all of the above. In particular, when working over the field of rational numbers, 
we
obtain a sheaf of commutative dg-algebras, which gives a solution to the problem of commutative cochains in the context of intersection cohomology.

The techniques developed in this paper have other potential applications. For instance, 
one could  also endow the functors of nearby and vanishing cycles with $E_\infty$-structures, as done in the last part of \cite{Na} for the rational case.
Also, Banagl's \cite{Banagl} theory of intersection space complexes and a
recent generalization of Banagl's theory due to
Agust\'{i}n and Fern\'{a}ndez de Bobadilla \cite{AgBo} should also allow for a treatment in the multiplicative setting.

\medskip
The paper is organized as follows.
 In Section \ref{sheaveseinfty} we develop the necessary tools of sheaf theory for $E_\infty$-algebras.
 In Section \ref{SecFil} we study filtered $E_\infty$-algebras. We define filtered versions of the cosimplicial normalization and study the interaction of Steenrod operations with spectral sequences.
 These results are applied in Section \ref{SecWeight}, where we define the weight filtration 
 at the cochain level. Lastly, Section \ref{SecInter} is devoted to intersection cohomology. 
 This last section can be read independently of Sections \ref{SecFil} and \ref{SecWeight}.

\section{Sheaves of $E_\infty$-algebras}\label{sheaveseinfty}
 In this section, we combine Mandell's \cite{ManCochains} cosimplicial normalization functor for $E_\infty$-algebras 
 together with Godement's \cite{Godement} cosimplicial resolution functor for sheaves of complexes  to 
 study sheaves of $E_\infty$-algebras and their cohomology.
 We define derived direct image and global sections functors
for sheaves of $E_\infty$-algebras and use them to describe the $E_\infty$-structure on
the complex of singular cochains of a topological space,
in terms of a resolution of its constant sheaf.

 \subsection{Preliminaries on $E_\infty$-algebras}
 Throughout this paper, $R$ will be a commutative ring. All operads considered will be in the category $\ch{R}$ of cochain 
 complexes of $R$-modules.
 
\begin{definition}
  Let $\Com$ denote the operad of commutative algebras, given in each arity $n\geq 0$, by $\Com(n)=R$ concentrated in degree 0,
 with $\Sigma_n$ acting by the identity.
 \begin{enumerate}[(1)]
  \item An operad $\Oo$ is called \textit{acyclic} if there is a map of operads $\Oo\to \Com$ such that in each arity $n$, the map $\Oo(n)\to \Com(n)$ is a quasi-isomorphism.
  \item An operad $\Oo$ is called \textit{$\Sigma$-free} if for each arity $n$, the underlying graded $R[\Sigma_n]$-module of $\Oo(n)$ is free in each degree.
  \item An \textit{$E_\infty$-operad} is a $\Sigma$-free acyclic operad 
  $\Oo\to \Com$ such that each graded module $\Oo(n)$ is concentrated in non-positive degrees.
 \end{enumerate}
\end{definition}

Throughout this paper we will let $\Ee$ be a cofibrant $E_\infty$-operad. 
Recall that being cofibrant means that $\Ee$ has the lifting property with respect to
maps of operads that are surjective quasi-isomorphisms (see for instance \cite{Hinich}).

Denote by $\Einf_R$ the category of $\Ee$-algebras over $R$. An object of $\Einf_R$ is 
given by a cochain complex together with an action of the $E_\infty$-operad $\Ee$.
We will mostly restrict to $\Ee$-algebras
whose underlying cochain complex is non-negatively graded.

The category $\Einf_R$ is complete
and cocomplete. Limits and filtered colimits
commute with the forgetful functor to cochain complexes.
We will denote by $\Qq$ the class of quasi-isomorphisms of $\Einf_R$-algebras (those morphisms of $\Ee$-algebras inducing quasi-isomorphisms at the cochain level) and by $$\Ho(\Einf_R):=\Einf_R[\Qq^{-1}]$$ the homotopy category
defined by formally inverting quasi-isomorphisms.

\subsection{Cosimplicial normalization functor for $E_\infty$-algebras}
We recall Mandell's construction of the cosimplicial normalization functor for $\Ee$-algebras \cite{ManCochains}. 

Consider the functor of normalized cochains, sending every cosimplicial cochain complex
to a cochain complex,
$$N:\Delta \ch{R}\lra \ch{R}.$$ This is given
by the total complex of the double complex obtained by normalizing
degree-wise. When dealing with unbounded complexes, the total complex is constructed using
the cartesian product, rather than the direct sum. 
The following description of the normalized cochains functor will be useful.
Denote by \[C_*(\Delta^n):=N(R[\Delta^n])\] the normalized chain complex of $R$-modules of the standard $n$-simplex simplicial set,
where \[R[-]:\Delta^{op}\mathrm{Set}\lra \Delta^{op} R\text{-}\mathrm{mod}\] denotes the free $R$-module functor, left adjoint to the forgetful.
We will often consider $C_*(\Delta^n)$ as a non-positively graded cochain complex.
Denote by $\Hom(-,-)$ the internal hom of cochain complexes. The following proposition is well-known:
\begin{prop}
If $A^{\bullet}$ is a cosimplicial cochain complex, then $N(A^{\bullet})$
may be identified with the end 
$$N(A^{\bullet})=\int_\alpha \Hom(C_*(\Delta^\alpha),A^{\alpha})$$
of the functor $\Delta^{op}\times \Delta\lra \ch{R}$
given by $(\alpha,\beta)\mapsto \Hom(C_*(\Delta^\alpha),A^{\beta})$.
\end{prop}

A key ingredient to promote the cosimplicial normalization functor from cochain complexes to $\Ee$-algebras is the 
Eilenberg-Zilber operad, which we recall next (see \cite{HiSc}, see also \cite{MaySheaf}).
\begin{definition}
The \textit{Eilenberg-Zilber operad} $\Zz=\End_\Lambda$ is the endomorphism operad  of the functor
 \[\Lambda:\Delta\lra \ch{R}\text{ defined by }
[n]\mapsto C_*(\Delta^n).\]
In each arity
$n\geq 0$,
$\Zz(n)$ is the normalization of the cosimplicial (non-positively graded) cochain complex  $C_*(\Delta^\bullet)^{\otimes n}$.
We may write:
$$\Zz(n)=\int_\alpha \Hom(C_*(\Delta^\alpha),C_*(\Delta^\alpha)^{\otimes n}).$$ 
In particular, we have $\Zz(0)\cong R$. 
The structure maps
$$\Zz(n)\otimes \Zz(j_1)\otimes \cdots \otimes \Zz(j_n)\lra \Zz(j_1+\cdots + j_n)$$
of $\Zz$ are given by the \textit{generalized Alexander-Whitney maps}
$$f\otimes(f_1\otimes \cdots \otimes f_n)\mapsto (-1)^{|f|\cdot \sum |f_i|} (f_1\otimes \cdots \otimes f_n)\circ f,$$
where $|g|$ denotes the degree of the map $g$.
\end{definition}

The main theorem of \cite{HiSc} (see also Theorem 5.5 of \cite{ManCochains}) is the following.
\begin{theorem}\label{algstrN}Let $\Oo$ be an operad and 
 $A^\bullet$ a cosimplicial $\Oo$-algebra.
Then $N(A^\bullet)$ is a $(\Oo\otimes\Zz)$-algebra, that is natural in maps of the
operad $\Oo$ and of the cosimplicial $\Oo$-algebra $A^\bullet$.
\end{theorem}
\begin{proof}
We explain the main steps for obtaining the algebra structure on $N(A^\bullet)$.
Let $\theta$ be the action maps of $\Oo$ on $A^\bullet$ and denote by
$$\alpha_n:\Zz(n)\otimes N(A^\bullet)\otimes\cdots \otimes N(A^\bullet)\lra N(\mathrm{diag}(A^\bullet\otimes \cdots \otimes A^\bullet))$$
the map given by 
$$f\otimes(a_1\otimes\cdots\otimes a_n))\mapsto (-1)^{|f|\cdot \sum |a_i|}(a_1\otimes \cdots\otimes a_n)\circ f.$$
Define action maps of $\Oo\otimes \Zz$ on $N(A^\bullet)$ by the composition
$$
\xymatrix{
\Oo(n)\otimes \Zz(n)\otimes N(A^\bullet)\otimes\cdots \otimes N(A^\bullet)\ar[d]^{Id_{\Oo(n)}\otimes \alpha_n}\\
\Oo(n)\otimes N(\mathrm{diag}(A^\bullet\otimes\cdots\otimes A^\bullet))\ar[d]^{\zeta}\\
N(\Oo(n)\otimes \mathrm{diag}(A^\bullet\otimes \cdots \otimes A^\bullet))\ar[d]^{N(\theta)}\\
N(A^\bullet),
}
$$
where $\zeta$ is induced by the natural morphism of cochain complexes 
$$\zeta:A\otimes \Hom(B,C)\to \Hom (B,A\otimes C)$$ given by 
$\zeta(x\otimes f)(y)=x\otimes f(y)$. One may verify that these maps 
satisfiy the axioms of associativity, unit and equivariance, and that they
are natural for maps of operads and of cosimplicial algebras.
\end{proof}

Since $\Zz$ is an acyclic operad,
by the lifting property of the cofibrant operad $\Ee$, 
we may choose a quasi-isomorphism of operads $\Ee\lra\Zz$. 
Furthermore, there is a map of operads $\Ee\to \Ee\otimes\Ee$ in such a way that both compositions 
$$\Ee\to \Ee\otimes\Ee\to \Cc om\otimes\Ee=\Ee\text{ and }\Ee\to \Ee\otimes\Ee\to \Ee\otimes\Cc om=\Ee$$
are the identity (see \cite[Lemma 5.7]{ManCochains}. This gives:

\begin{definition}
 The \textit{cosimplicial normalization functor} 
 $$N_\Ee:\Delta\Einf_R\lra \Einf_R$$
 is defined by $N$ composed with the reciprocal image functor of the composition
 $$\Ee\lra \Ee\otimes\Ee\lra \Ee\otimes\Zz.$$
\end{definition}

The functor $N_\Ee$ satisfies the following properties:
\begin{enumerate}
 \item By forgetting the algebra structures we recover the normalized cochains functor $N$.
 \item For a constant cosimplicial $\Ee$-algebra $A^\bullet$ (in which all face and degeneracy maps are identities),
 the isomorphism of cochain complexes $A^0\cong N(A^\bullet)$ is a morphism of $\Ee$-algebras.
\end{enumerate}

We refer to \cite{ManCochains} for proofs of the above facts.

\subsection{Derived direct image and global sections functors}\label{Derivedfunctors}

Denote by $\Sh_X(\ch{R})$ the category of sheaves of cochain complexes of $R$-modules on a topological space $X$.
It is a symmetric monoidal category, with the product of two sheaves $\Aa$ and  $\Bb$
given by the sheafification of the presheaf $U\mapsto \Aa(U)\otimes \Bb(U)$. The unit is the constant sheaf $\underline{R}_X$.

\begin{definition}The \textit{cosimplicial Godement
resolution} of a sheaf $\Aa\in \Sh_X(\chb{\geq 0}{R})$ is an augmented cosimplicial sheaf $G^\bullet(\Aa)\in \Delta \Sh_X(\chb{\geq 0}{R})$
satisfying:
\begin{enumerate}[(i)]
 \item $G^{-1}(\Aa)=\Aa$,
 \item the map $\Aa\to \mathbf{s}(G^+(\Aa))$ is a quasi-isomorphism, and
 \item $\mathbf{s}G^+(\Aa)$ is a complex of flasque sheaves on $X$,
\end{enumerate}
where $\mathbf{s}:\Delta \ch{R}\lra \ch{R}$
denotes the total simple functor, defined by the total complex of the double complex
obtained without normalizing.
\end{definition}
The functor $\mathbf{s}$ is naturally homotopy equivalent
to the cosimplicial normalization functor $N$ (see for instance \cite[Theorem 6.1]{MacLane}).
We refer to \cite[Appendix]{Godement} for the construction and properties of $G^\bullet$.
The assignment $\Aa\mapsto G^\bullet(\Aa)$ defines a lax symmetric monoidal functor (see for instance \cite[Proposition 3.12]{RR}). In particular,
it induces a cosimplicial Godement resolution functor 
\[G^\bullet:\Sh_X(\Einf_R)\lra \Delta \Sh_X(\Einf_R)\]
on sheaves of $\Ee$-algebras.

Let $f:X\to Y$ be a continuous map of topological spaces.
\begin{definition}
The \textit{direct image functor} \[f_*:\Sh_X(\ch{R})\lra \Sh_Y(\ch{R})\]
is defined by 
\[f_*\Aa(U):=\Aa(f^{-1}(U)).\] 
The \textit{inverse image functor} 
\[f^*:\Sh_Y(\ch{R})\lra \Sh_X(\ch{R})\]
is defined by taking $f^*\Bb$ to be the sheafification of the presehaf 
\[U\mapsto \varinjlim_{V\supseteq f(U)}\Bb(V),\]
for any sheaf $\Bb$ on $Y$.

\end{definition}
The direct image functor $f_*$ is left exact, but not right exact in general.
The inverse image functor $f^*$ is exact and left adjoint to $f_*$.
In particular, there are natural adjunction morphisms
$\Bb\to f_*f^*\Bb$ and $f^*f_*\Aa\to \Aa$.

\begin{lemma}\label{fstarlsm}
The functor $f_*$ is lax symmetric monoidal.
\end{lemma}
\begin{proof}Define a map
$\varepsilon:\underline{R}_Y\to f_*\underline{R}_X$
via the adjunction morphism
$$\underline{R}_Y\to f_*f^*\underline{R}_Y=f_*\underline{R}_X.$$ 
This makes $f_*$ compatible with the unit.
We next define a natural transformation 
$$\mu_{\Aa,\Bb}:f_*\Aa\otimes f_*\Bb\lra f_*(\Aa\otimes\Bb).$$
Denote by
$\Gg$ the presheaf defined by $\Gg(U):=\Aa(U)\otimes \Bb(U)$.
Consider the diagram of presheaves
$$
\xymatrix{
\ar[d]_-{f_*\theta}f_*\Gg\ar[r]^{\theta'}&a(f_*\Gg)\ar@{.>}[dl]^\mu\\
f_*aG
}
$$
where $\theta:\Gg\to a\Gg$ and $\theta':f_*\Gg\to a(f_*\Gg)$ denote the sheafification maps.
The universal property of $\theta'$ ensures
that there is a unique dotted map
$\mu$ making the diagram commute.
Then:
$$f_*\Aa\otimes f_*\Bb=a(f_*\Gg)\stackrel{\mu}{\lra} f_*(a\Gg)= f_*(\Aa\otimes\Bb).\qedhere$$
\end{proof}

\begin{remark}\label{rmkadjmor}
The above lemma
gives a direct image functor
$$f_*:\Sh_X(\mathbf{E}_R)\lra \Sh_Y(\mathbf{E}_R)$$
for sheaves of $\Ee$-algebras.
Note as well that $f^*$ is strong monoidal, since $\underline{R}_X\cong f^*\underline{R}_Y$
and there are natural isomorphisms 
\[f^*(\Aa\otimes\Bb)\cong f^ *\Aa\otimes f^*\Bb.\]
Therefore $f^*$ induces an inverse image functor 
\[f^*:\Sh_Y(\mathbf{E}_R)\lra \Sh_X(\mathbf{E}_R)\]
for sheaves of $\Ee$-algebras. Furthermore, the natural adjunction
morphisms 
\[\Bb\to f_*f^*\Bb\text{ and }f^*f_*\Aa\to \Aa.\]
are morphisms in $\Sh_Y(\mathbf{E}_R)$ and $\Sh_X(\mathbf{E}_R)$ respectively.
\end{remark}

\begin{definition}
Define the \textit{derived direct image functor}
$$\RR_\Ee f_*:\Sh_X(\Einf_R)\lra \Sh_Y(\Einf_R)$$
via the composition of functors
$$\Sh_X(\Einf_R)\stackrel{G^+}{\lra}\Delta \Sh_X(\Einf_R)\xra{f_*}\Delta \Sh_Y(\Einf_R)\stackrel{N_\Ee}{\lra}\Sh_Y(\Einf_R)$$
where $N_\Ee$ is defined on a cosimplicial sheaf $\Aa^\bullet$ via Mandell's cosimplicial normalization functor:
\[N_\Ee(\Aa^\bullet)(U):=N_\Ee(\Aa^\bullet(U)).\]
\end{definition}

Note that, at the level of the underlying cochain complexes, 
the functor $\RR_\Ee f_*$ is naturally quasi-isomorphic to the additive
derived direct image functor
$$\RR f_*=\mathbf{s}\circ  f_*\circ G^+\simeq N\circ  f_*\circ G^+:\Sh_X(\ch{R})\lra \Sh_Y(\ch{R}).$$

For every continuous map $f:X\to Y$ of topological spaces, the functor $\RR_\Ee f_*$ preserves quasi-isomorphisms. 
Hence it induces a functor
$$\RR_\Ee f_*:\Ho(\Sh_X(\Einf_R))\lra \Ho(\Sh_Y(\Einf_R))$$
at the level of homotopy categories, which might also be called derived categories.
The following is straightforward:
\begin{lemma}\label{co_unit}
There is a unit map \[1\lra \RR_\Ee f_* \circ f^* \text{ in }\Sh_Y(\Einf_R)\] and a counit map 
\[f^*\circ \RR_\Ee f_*\stackrel{\sim}{\longleftarrow}f^*\circ f_*\lra 1\text{ in }\Ho(\Sh_X(\Einf_R)).\]
In particular, $(f^*,\RR_\Ee f_*)$ is an adjoint pair at the level of homotopy categories.
\end{lemma}

\begin{lemma}\label{compoderived}
Let $f:X\to Y$ and $g:Y\to Z$ be continuous maps of 
topological spaces. There is a natural transformation of functors
\[\RR_\Ee (g\circ f)_*\Longrightarrow \RR_\Ee g_*\circ \RR_\Ee f_*\]
which is a quasi-isomorphism.
\end{lemma}
\begin{proof}
Note that $N_\Ee\circ g_*=g_*\circ N_\Ee$. Indeed, for any cosimplicial sheaf $\Aa^\bullet$ and any open subset $U\subseteq X$, we have
$$N_\Ee(g_*\Aa^\bullet)(U)=N_\Ee(g_*\Aa^\bullet(U))=N_\Ee(\Aa^\bullet(g^{-1}(U))=g_*N_\Ee\Aa^\bullet(U).$$
Furthermore, there is a natural transformation $g_*\Rightarrow \RR_\Ee g_*$ which is a quasi-isomorphism. Indeed, for any
sheaf $\Aa$, the natural map $\Aa\to G^+(\Aa)$ gives a quasi-isomorphism 
$$\Aa\cong N_\Ee(\Aa)\longrightarrow N_\Ee\circ G^+(\Aa).$$
All together gives:
$$\RR_\Ee(g\circ f)_*= N_\Ee\circ g_*f_*\circ G^+ \Rightarrow g_*\circ N_\Ee\circ f_*\circ G^+=g_*\circ \RR_\Ee f_*\Rightarrow \RR_\Ee g_*\circ \RR_\Ee f_*.\qedhere$$
\end{proof}

Consider the global sections functor $\Gamma(X,-):=(\rho_X)_*$ where $\rho_X:X\to \ast$ is the map to the point.

\begin{definition}
Define the \textit{derived global sections functor}
$$\RR_\Ee\Gamma(X,-):\Sh_X(\Einf_R)\lra \Einf_R$$
via the composition of functors
$$\Sh_X(\Einf_R)\stackrel{G^+}{\lra}\Delta \Sh_X(\Einf_R)\xra{\Gamma(X,-)}\Delta \Einf_R\stackrel{N_\Ee}{\lra}\Einf_R.$$
\end{definition}
Note that $\RR_\Ee\Gamma(X,-)=\RR_\Ee(\rho_X)_*$.

\begin{lemma}\label{relativesheaf}
 Let $f:X\to Y$ be a continuous map of 
 topological spaces and $\Aa$ a sheaf of $\Ee$-algebras on $X$. 
 Then there is a quasi-isomorphism of $\Ee$-algebras $$\RR_\Ee\Gamma(X,\Aa)\longrightarrow \RR_\Ee\Gamma(Y,\RR_\Ee f_*\Aa).$$
\end{lemma}
\begin{proof}
It follows from Lemma \ref{compoderived}, since
$\RR_\Ee\Gamma(X,-)=\RR_\Ee(\rho_X)_*,$ where $\rho_X:X\to *$ is the morphism to the point.
\end{proof}

\subsection{Comparison with singular cochains}
The singular cohomology of Hausdorff, paracompact and locally contractible topological 
spaces (such as manifolds or analytic spaces) may be computed via the cohomology of the constant sheaf
$$H^i(X,R)\cong H^i(X,\underline{R}_X).$$
We next refine this result to a statement at the level of $\Ee$-algebras.

\begin{lemma}\label{quistocochains}Let $X$ be a Hausdorff, paracompact and locally contractible topological space and let
$\Aa\in \Sh_X(\Einf_R)$. Assume that for each $x\in X$,
the unit map $R\to \Aa_x$ is a quasi-isomorphism of $\Ee$-algebras.
Then:
\begin{enumerate}
 \item \label{cons1}The sheaf $\Aa$ is quasi-isomorphic as a sheaf of $\Ee$-algebras to the sheafification $aC^*(-,R)$ 
 of the presheaf of singular cochains $C^*(-,R)$.
 \item \label{cons2}There is a string of quasi-isomorphisms of $\Ee$-algebras
 from $\RR_\Ee\Gamma(U,\Aa)$ to $C^*(U,R)$, for every open subset $U\subseteq X$.
\end{enumerate}
\end{lemma}
\begin{proof}Define morphisms of presheaves
$$\Aa\to \Aa\otimes C^*(-,R)\text{ and }C^*(-,R)\to \Aa\otimes C^*(-,R)$$
by letting $f\mapsto f\otimes 1$ and $c\mapsto 1\otimes c$ respectively. After sheafification, 
these give morphisms of sheaves of $\Ee$-algebras
$$\Aa\lra a(\Aa\otimes C^*(-,R))\longleftarrow aC^*(-,R)$$
which by assumption, induce quasi-isomorphisms at their stalks. Since $X$ is locally contractible,
it follows that these morphisms are global quasi-isomorphisms and $(\ref{cons1})$ is satisfied.

We now prove $(\ref{cons2})$. Recall that when forgetting the multiplicative structures, 
the functor $\RR_\Ee\Gamma(X,-)$ is naturally quasi-isomorphic to the 
ordinary derived functor of global sections in the additive setting.
In particular, for any sheaf $\Aa$ of $\Ee$-algebras and any open subset $U\subseteq X$, its hypercohomology may be computed 
by 
$$H^n(U,\Aa)=H^n(\RR_\Ee\Gamma(U,\Aa)).$$ 
Therefore we have a string of quasi-isomorphisms of $\Ee$-algebras
$$\RR_\Ee\Gamma(U,\Aa)\lra \RR_\Ee\Gamma(U,a(\Aa\otimes C^*(-,R)))\longleftarrow \RR_\Ee\Gamma(U,aC^*(-,R))\longleftarrow C^*(U,R),$$
where for the last arrow, we used the fact that the 
sheafification map from $C^*(U,R)$ to $aC^*(U,R)$ is a quasi-isomorphism of $\Ee$-algebras.
\end{proof}

\begin{theorem}\label{quistocochains_thm}
The assignment $X\mapsto \RR_\Ee\Gamma(X,\underline{R}_X)$, where  $\underline{R}_X$ denotes the constant sheaf on $X$, 
defines a contravariant functor $\mathrm{Top}\lra \Einf_R$ 
from the category $\mathrm{Top}$ of Hausdorff, paracompact and locally contractible topological spaces
to the category of $\Ee$-algebras,
which is naturally quasi-isomorphic to 
the functor $C^*(-,R)$ defined by the complex of singular cochains with its $\Ee$-algebra structure.
\end{theorem}
\begin{proof}
We first show that this construction is functorial. 
The proof is similar to the additive setting.
Let $f:X\to Y$ be a continuous map of topological spaces.
Then we have a map $f^*G^+\underline{R}_Y\to G^+\underline{R}_X$ (see for example \cite[Proposition 5.2.1]{Huber}).
By composing with the adjunction morphism, we get
$$G^+\underline{R}_Y\to f_*f^*G^+\underline{R}_Y\to f_*G^+\underline{R}_X.$$
Therefore we have
\begin{align*}
\RR_\Ee\Gamma(Y,\underline{R}_Y)=N_\Ee\circ (\rho_Y)_* \circ G^+(\underline{R}_Y)\to 
N_\Ee\circ (\rho_Y)_* f_* \circ G^+(\underline{R}_X)=\\= N_\Ee\circ (\rho_X)_* \circ G^+(\underline{R}_X)=\RR_\Ee\Gamma(X,\underline{R}_X).
\end{align*}
Compatibility with composition of maps follows from the fact that the adjunction morphism is compatible with composition.
This proves that 
\[\RR_\Ee(-,\underline{R}_{-}):\mathrm{Top}\lra \Einf_R\]
is a contravariant functor.
It now suffices to apply Lemma \ref{quistocochains} to the constant sheaf.
\end{proof}

\begin{remark}\label{remTW}
Theorem \ref{quistocochains_thm} implies that the functor
 $X\mapsto \RR_\Ee\Gamma(X,\underline{R}_X)$ 
 is a cochain theory in the sense of Mandell \cite{ManCochains}.
For $R=\QQ$, Sullivan's functor of piece-wise linear forms $A_{PL}(-)$
 defines a cochain theory. In particular, 
the  
$\Ee$-algebra $\RR_\Ee\Gamma(X,\underline{\QQ}_X)$
is naturally quasi-isomorphic 
to $A_{PL}(X)$. Also, in \cite{Na}, Navarro-Aznar defined a commutative dg-algebra 
$\RR_{\TW}\Gamma(X,\underline{\QQ}_X)$ quasi-isomorphic to $A_{PL}(X)$
using the Thom-Whitney simple functor
\[\mathbf{s}_{\TW}:\Delta\mathrm{CDGA}_\QQ\lra \mathrm{CDGA}_\QQ.\]
Over the rationals, the two constructions are naturally quasi-isomorphic.
In \cite{Petersen}, Petersen constructed a global sections functor for sheaves of $E_\infty$-algebras, using 
the theory of twisted algebras, which is also valid for the compact support case.
\end{remark}

\section{Filtered $E_\infty$-algebras}\label{SecFil}
In this section we review the theory of filtered $E_\infty$-algebras. We define filtered normalization functors which depend on the filtration that one considers on the operad $\Ee$ and describe the behaviour of the normalization functor with respect to spectral sequences.
We also study Steenrod and reduced $\ell$-th power operations on spectral sequences.

\subsection{Filtered operads and filtered algebras}
Denote by $\ch{\Ff R}$ the category of filtered cochain complexes of $R$-modules.
Recall that given filtered complexes $(A,W)$ and $(B,W')$, then their tensor product is again a filtered complex, with the filtration
\[(W*W')_p(A\otimes B):=\sum_{i+j=p} \Img(W_iA\otimes W_j'B\lra A\otimes B).\]
This makes $\ch{\Ff R}$ into a symmetric monoidal category, where the unit is given by $R$ with the trivial filtration
$0=t_{-1}R\subset t_0R=R$.
Also, the internal hom of complexes is a filtered cochain complex, with the filtration
$$(W|W')_p\Hom^n(A,B):=\{f:A\to B; f(W_qA^m)\subset W'_{q+p}B^{m+n}\}.$$

Let us recall three functorial increasing filtrations that will play important roles in the sequel:
\begin{definition}
Let $A$ be a cochain complex. The \textit{trivial filtration} $t$ on $A$ is 
$$0=t_{-1}A^n\subset t_0A^n=A^n\text{ for all }n\in\ZZ.$$
The \textit{canonical filtration} $\tau$ on $A$ is given by 
$$
\tau_pA^n=\left\{ 
\begin{array}{ll}
0&,p<n\\
\Ker(d:A^n\to A^{n+1})&,p=n\\
A^n&,p>n
\end{array}
\right..
$$
The \textit{b\^{e}te filtration} $\sigma$ is given by
$\sigma_pA^n=A^n$ if $n\geq -p$ and $\sigma_pA^n=0$ if $n< -p.$
\end{definition}
The three filtrations $t$, $\tau$ and $\sigma$ define lax symmetric monoidal functors from
complexes to filtered complexes.
We will use the following easy lemma:
\begin{lemma}\label{maptosigma}
Let $A$ be a cochain complex concentrated in non-positive degrees. Then the identity defines 
morphisms of filtered complexes 
$$(A,\sigma)\lra (A,t)\lra (A,\tau).$$
\end{lemma}
\begin{proof}
For the first map, it suffices to note that if $n\geq 0$, then $\sigma_pA^{-n}$ is non-trivial only when $n\leq p$, for which $t_pA^{-n}=A^{-n}$.
For the second map, note that $\tau_pA^n=A^n$ for all $(p,n)$ such that $p\geq 0$ and $n\leq 0$. Indeed, if
$p=n=0$ then $\tau_pA^n=\Ker(d:A^0\to A^1)=A^0$, while if $p>n$ then $\tau_pA^n=A^n$.
\end{proof}

This triad of filtrations, $t$, $\sigma$ and $\tau$
already plays a key role in Deligne's additive mixed Hodge theory
and will be essential for the applications to the weight filtration of Section \ref{SecWeight}.
In this section, we will focus on the filtrations $t$ and $\sigma$ defined on the operad $\Ee$. 

Recall that, given a filtered complex $(A,W)$, where $W$ is an increasing filtration, the $r$-th page $E_r$ of its associated spectral sequence is given by:
\[E_r^{-p,q}(A,W):={{\{x\in W_{p}A^{q-p};dx\in W_{p-r}A^{q-p+1}\}}\over{\{x\in W_{p-1}A^{q-p}; dx\in W_{p-r}A^{q-p+1}\}+dW_{p+r-1}\cap W_pA^{q-p}}}.\]
The differential on $E_r$, induced by the differential on $A$, has bidegree $(r,-r+1)$,
\[d_r:E_r^{-p,q}(A,W)\lra E_r^{-p+r,q-r+1}(A,W),\]
making $(E_r^{*,*}(A,W),d_r)$ into an $r$-bigraded cochain complex.

\begin{definition}Let $r\geq 0$. 
A morphism $f:(A,W')\to (B,W')$ of filtered complexes is called \textit{$E_r$-quasi-isomorphism} if it
induces a quasi-isomorphism of bigraded complexes at the $E_r$-stage of the associated spectral sequences, so that the map
$$E_{r+1}(f):E_{r+1}(A,W')\to E_{r+1}(B,W')$$ is an isomorphism of bigraded $R$-modules. 
\end{definition}

We will denote by $\Qq_r$ the class of $E_r$-quasi-isomorphisms. 
We have a chain of inclusions 
\[\cdots\subseteq\Qq_{r+1}\subseteq \Qq_r\subseteq\cdots\subseteq \Qq.\]

\begin{definition}
A \textit{filtered operad} $(\Oo,W)$ is an operad in $\ch{\Ff R}$. 
Alternatively, it is given by an operad $\Oo$ in $\ch{R}$ together with a filtration
$W$ on each cochain complex $\Oo(n)$ in such a way that its structure morphisms 
$$\Oo(\ell)\otimes \Oo(m_1) \otimes \dots \otimes \Oo(m_\ell)
\longrightarrow \Oo(m_1 + \cdots + m_\ell)$$
are compatible with filtrations and
its unit satisfies $\eta : R \longrightarrow W_0\Oo(1)$.
\end{definition}

Note that if $(\Oo,W)$ is a filtered operad, then 
for all $r\geq 0$, the collection
$$E_r^{*,*}(\Oo,W):=\{E_r^{*,*}(\Oo(\ell),W)\}$$
defines an operad in $r$-bigraded cochain 
complexes (see also \cite[Proposition 5.1]{KSV}).

\begin{definition}
Let $(\Oo,W)$ be a filtered operad. A \textit{filtered $(\Oo,W)$-algebra} is a filtered cochain complex $(A,W')$
together with an action of $\Oo$ on $A$ whose structure morphisms
$$\theta_A(\ell): \Oo(\ell)\otimes_{\Sigma_\ell} A^{\otimes \ell} \longrightarrow A$$
are compatible with filtrations.
\end{definition}

We refer to \cite{Brunoetal} for an equivalent treatment of filtered algebras via the endomorphism operad.
Note that if $(A,W')$ is a filtered $(\Oo,W)$-algebra, then for all $r\geq 0$,
we have that $E_r(A,W')$ is an $E_r(\Oo,W)$-algebra.

\subsection{Filtered $\Ee$-algebras}
For the study of filtered $\Ee$-algebras, an obvious choice for a filtration on the operad $\Ee$ is the trivial filtration.
This choice will be adequate if one is interested in the class $\Qq_0$ of $E_0$-quasi-isomorphisms.
Indeed, for the trivial filtration $t$, we have 
\[E_0^{-p,q}(\Ee,t):=Gr^t_{p}\Ee^{q-p}\cong 
\left\{ 
\begin{array}{ll}
\Ee^{q}, &\text{ if }p=0\\
0, &\text{ otherwise }
\end{array}\right.\]
and so $E_0^{*,*}(\Ee,t)$ is a cofibrant $E_\infty$-operad in the category of $0$-bigraded complexes.
In addition, we have $E_1^{0,0}(\Ee,t)\cong \Cc om$ and $E_1^{*,*}(\Ee,t)$ is trivial in the remaining bidegrees.
As a consequence, if $(A,W)$ is a filtered $(\Ee,t)$-algebra then $E_0^{*,*}(A,W)$ is an $\Ee$-algebra (in $0$-bigraded complexes)
and $E_1^{*,*}(A,W)$ is a commutative algebra (in $1$-bigraded complexes).

If one is interested in the class $\Qq_r$ of $E_r$-quasi-isomorphisms, with $r>0$, one is led to shift the trivial filtration in order to obtain isomorphisms
\[E_{r+1}^{0,0}(\Ee(\ell))\cong \Cc om(\ell)\text{ and }E_{r+1}^{*,*}(\Ee(\ell))=0\text{ for }(p,q)\neq(0,0).\]
In the case $r=1$, this is achieved by the b\^{e}te filtration. Indeed, we have
\[E_1^{-p,q}(\Ee,\sigma):=H^{q-p}(Gr^\sigma_{p}\Ee)\cong 
\left\{ 
\begin{array}{ll}
\Ee^{-p},& \text{ if }q=0\\
0,& \text{ otherwise }
\end{array}\right..\] 
Therefore $E_1^{*,*}(\Ee,\sigma)$ is a cofibrant $E_\infty$-operad in the category of $1$-bigraded complexes.
Also, $E_2^{0,0}(\Ee,\sigma)\cong  \Cc om$ and $E_2^{*,*}(\Ee,\sigma)$ is trivial otherwise.
As a consequence, if $(A,W)$ is a filtered $(\Ee,\sigma)$-algebra then $E_1^{*,*}(A,W)$ is an $\Ee$-algebra
and $E_2^{*,*}(A,W)$ is a commutative algebra.

\begin{remark}
By Lemma \ref{maptosigma} we have a morphism of filtered operads
\[(\Ee,\sigma)\lra (\Ee,t)\]
and so if $(A,W)$ is an $(\Ee,t)$-algebra then it may also be viewed as an $(\Ee,\sigma)$-algebra.
\end{remark}

We will denote by $\Ff\Einf^t_R$ the category of finitely filtered non-negatively graded $(\Ee,t)$-algebras
and by \[\Ho(\Ff\Einf^t_R):=\Ff\Einf^t_R[\Qq_0^{-1}]\]
        the localized category with respect to $E_0$-quasi-isomorphisms.
        Likewise, let $\Ff\Einf^\sigma_R$ denote the category of $(\Ee,\sigma)$-algebras
        and 
\[\Ho(\Ff\Einf^\sigma_R):=\Ff\Einf^t_R[\Qq_1^{-1}]\]
  the localized category with respect to $E_1$-quasi-isomorphisms.

  \begin{remark}
   Given any filtration $W$ on a cochain complex $A$, one may consider the shifted filtration 
$SW_pA^n:=W_{p+n}A^n$
which satisfies 
\[E_{r+1}^{p,q}(A,SW)\cong E_1^{-q,p+2q}(A,W)\quad\text{for all}\quad r\geq 0.\]
In particular, $\sigma=St$ is just a shift of the trivial filtration. For all $r\geq 0$ we may consider the $r$-shift 
$S^rt$
which serves as a natural  filtration on $\Ee$ for studying the category of 
filtered $\Ee$-algebras with respect to $E_r$-quasi-isomorphisms. We have chosen to restrict to the cases $r=0$ and $r=1$ only, 
which are the cases arising in the study of the weight filtration.
  \end{remark}

\subsection{Normalization functors for filtered algebras}\label{filnorm}

We will consider an increasing filtration $F=\{F_p\}$ on $C_*(\Delta^\bullet)$, 
compatible with the differential and the cosimplicial structure.
Then, given a cosimplicial filtered complex $A^\bullet$, the cochain complex 
$N(A^\bullet)$ inherits a natural filtration from that in $\Hom(C_*(\Delta^\bullet),A^\bullet)$. 
This gives a filtered normalization functor $N^F$ for filtered complexes.

\begin{lemma}\label{filteredZOalg}Let $(\Oo,W)$ be a filtered operad and
$(A^{\bullet},W')$ a cosimplicial filtered $(\Oo,W)$-algebra. Let $F$ be a filtration on 
$C_*(\Delta^\bullet)$.
Then 
$$N^F(A^{\bullet},W'):=\int_\alpha \Hom_{\ch{\Ff R}}((C_*(\Delta^\alpha),F), (A^{\bullet},W))
$$
is a filtered $(\Oo\otimes \Zz,W*F)$-algebra and the filtration on $N(A^{\bullet})$ is induced by 
the filtration $(F|W')$ on 
$\Hom(C_*(\Delta^\alpha),A^{\alpha})$.
\end{lemma}
\begin{proof}
It suffices to show that the action maps given in the proof of 
Theorem \ref{algstrN} are compatible with filtrations:
$$W_p\Oo(n)\otimes F_q\Zz(n)\otimes W_{p_1}N(A^\bullet)\otimes\cdots \otimes W_{p_n}N(A^\bullet)\lra W_{p+q+p_1+\cdots+p_n}N(A^\bullet).$$
We first show that
$$\alpha_n(F_q\Zz(n)\otimes W_{p_1}N(A^\bullet)\otimes\cdots \otimes W_{p_n}N(A^\bullet))\subset W_{p_1+\cdots + p_n}N(\mathrm{diag}(A^\bullet\otimes\cdots \otimes A^\bullet)).$$
This follows from the two easy facts that: first,
given $a_i\in W_{p_i}\Hom(A_i,B_i)$, then 
$$a_1\otimes \cdots \otimes a_n\in W_{p_1+\cdots + p_n}\Hom(A_1\otimes \cdots \otimes A_n,B_1\otimes \cdots \otimes B_n).$$
Second, given morphisms $f\in W_p\Hom(A,B)$ and $g\in W_q\Hom(B,C)$, then their composition satisfies 
$g\circ f\in W_{p+q}\Hom(A,C)$.
Next, note that the morphism 
$$\zeta:A\otimes\Hom(B,C)\to \Hom(B,A\otimes C)\text{ ; }\zeta(x\otimes f)(y)=x\otimes f(y)$$ is compatible with filtrations. Indeed, assume that $x\in W_iA$ and $f\in W_j\Hom(B,C)$.
Then, for $y\in W_qB$ we have $x\otimes f(y)\in W_{i+j+q}(A\otimes C)$. This proves that 
$$\zeta(x\otimes f)\in W_{i+j}\Hom(B,A\otimes C).$$
Lastly, since $\theta$ is compatible with filtrations and $N$ is functorial in the category of filtered complexes,
it follows that $N(\theta)$ is compatible with filtrations.
\end{proof}

Let $F$ be one of the filtrations $t$ (trivial) or $\sigma$ (b\^{e}te).
If $(A^\bullet,W)$ is a filtered $(\Ee,F)$-algebra, the above result makes $N^F(A^\bullet,W)$ into an 
$(\Ee\otimes\Zz,F*F)$-algebra. 
By composing $N^F$ with the reciprocal image functor of the filtered diagonal map
\[(\Ee,F)\to (\Ee\otimes \Ee,F*F)\to (\Ee\otimes\Zz,F*F)\]
we obtain a cosimplicial normalization functor $N_\Ee^F$ for filtered $(\Ee,F)$-algebras.

To study the weight spectral sequence it will be useful to understand the behavior of the b\^{e}te filtration
of the cosimplicial normalization functor.
In the additive setting this corresponds to the \textit{diagonal filtration} introduced in 7.1.6 of \cite{DeHIII} for filtered cochain complexes
(see also Section $\S$2 of \cite{CG3}). 

\begin{lemma}\label{E1lemma}
Let $(A^{\bullet},W)$ be a cosimplicial filtered $(\Ee,\sigma)$-algebra. The spectral sequence associated
with $N^\sigma_\Ee(A^\bullet,W)$ satisfies
\[E_1^{-p,q}(N_\Ee^\sigma(A^\bullet,W))
\cong E_1(\varphi)^*\int_\alpha \sum_m \Hom\left(C_m(\Delta^\alpha), E_1^{m-p,q}(A^\alpha,W) \right).
\]
where $\varphi:(\Ee,\sigma)\lra (\Ee\otimes\Zz,\sigma*\sigma)$.
\end{lemma}
\begin{proof}
Since the b\^{e}te filtration $\sigma$ satisfies $Gr_m^\sigma C_n(\Delta^\alpha)=0$ for all $n\neq m$, we obtain
$$Gr_{p}^{\sigma|W}\Hom^{q-p}(C_*(\Delta^\alpha),A^\alpha)\cong \sum_m \Hom(C_m(\Delta^\alpha),Gr_{p-m}^W(A^\alpha)^{m-p+q}).$$
Computing the cohomology of this graded complex, we get 
$$E_1^{-p,q}(\Hom^*(C_*(\Delta^\alpha), A^\alpha),\sigma|W)
 \cong \sum_m \Hom(C_m(\Delta^\alpha), E_1^{m-p,q}(A^\alpha,W)).
 $$
 Noting that $E_1$ commutes with $\varphi^*$, we get the desired isomorphism.
\end{proof}

\begin{remark}
 The above result just states that $N_\Ee^\sigma$ commutes with the functor $E_1$, when considering the obvious normalization functor for cosimplicial $\Ee$-algebras in 1-bigraded complexes. Similarly, for the trivial filtration, one analogously shows that $N_\Ee^t$ commutes with $E_0$.
\end{remark}

\subsection{Steenrod operations and spectral sequences}\label{SectionSteenrod}

Steenrod operations in spectral sequences have been studied in \cite{MaySteenrod}, \cite{SingerSS}, \cite{Dwyer} and \cite{TurnerSS} among others. In particular, Singer \cite{SingerSS} defined Steenrod operations on certain first quadrant spectral sequences that include the Leray-Serre spectral sequence
as an example and Dwyer \cite{Dwyer} dealt with the second quadrant case. The approach in both cases is to construct the operations directly at the level of the filtered complexes and then keep track of
how these operations behave in the spectral sequence.
In this section, we briefly review how Steenrod operations and reduced $\ell$-th power operations are defined on an $E_\infty$-algebra and on the spectral sequence of a filtered $(\Ee,\sigma)$-algebra.

Choosing $\Ee$ to be the Barratt-Eccles operad, we identify $\Ee(\ell)$ with a free resolution of $\FF_\ell$ as a $\Sigma_\ell$-module. Then, we may choose a family of elements $e_i\in \Ee(\ell)^{-i}$ such that if $A$ is an $\Ee$-algebra over $\FF_\ell$ with structure morphisms $\theta$, these give the operations 
\[D_i:A^{k_1}\otimes\cdots\otimes A^{k_\ell}\longrightarrow A^{k_1+\cdots+k_\ell-i}\] given by \[
D_i(a_1,\cdots,a_\ell):=\theta(e_i\otimes a_1\otimes\cdots\otimes a_\ell).\]
We distinguish two cases:
\begin{enumerate}
 \item 
If $\ell=2$, the operation $D_i(a,b)=a\cup_i b$ is just the cup$_i$-product of $a$ and $b$ and the
 \textit{Steenrod operations at the cochain level}
\[P^s:A^k\lra A^{k+s}\,;\, s\leq k\]
are defined by 
$$P^s(a):=da\cup_{k-s+1} a+ a\cup_{k-s}a=\theta(e_{k-s+1}\otimes da\otimes a+e_{k-s}\otimes a\otimes a).$$
Since $dP^s=P^sd$, these induce \textit{Steenrod operations in cohomology}
$$Sq^s:H^k(A)\to H^{k+s}(A)\,;\, Sq^s([a]):=[a\cup_{k-s}a]\text{ for all }s\leq k.$$
\item 
For $p>2$ we have \textit{reduced $\ell$-th power operations} 
\[P^s:A^k\to A^{k+2s(\ell-1)}\quad\text{ and }\quad \beta P^s:A^k\to A^{k+2s(\ell-1)+1}\]
defined via the operations $D_i$ applied to $\ell$ elements. These
induce operations in cohomology 
\[P^s:H^k(A)\to H^{k+2s(\ell-1)}(A)\quad\text{ and }\quad \beta P^s:H^k(A)\to H^{k+2s(\ell-1)+1}(A).\]
We refer to \cite{MaySteenrod} for the construction and main properties of these operations.
\end{enumerate}

In the filtered case, we have:
\begin{lemma}\label{myoperations}
 Let $(A,W)$ be a filtered $(\Ee,\sigma)$-algebra over $\FF_\ell$. Then we have operations
 \[P^s:E_r^{p,q}(A,W)\lra E_r^{p-q+s,2q}(A,W)\quad\text{for }\ell=2,\]
and 
\[P^s:E_r^{p,q}(A,W)\lra E_r^{p+(2s-q)(\ell-1),\ell q}(A,W)\quad\text{for }\ell>2.\]
when $\ell>2$, for $r=1,2$.
\end{lemma}
\begin{proof}
 Since $(A,W)$ is a filtered $(\Ee,\sigma)$-algebra, it follows that $E_1(A,W)$ is an $\Ee$-algebra in $1$-bigraded complexes and the above constructions give operations in $E_1$. Since these operations are compatible with $d_1$, they induce operations in $E_2$ which are well-defined.
\end{proof}

\begin{remark}\label{verticals}
Using Lemma \ref{E1lemma},
by results of May \cite{MaySteenrod}, Singer \cite{SingerSS} and Dwyer \cite{Dwyer} (see also Theorem 6.15 of \cite{McC}),
one may also define \textit{vertical operations}
\[P^s_v:E_1^{p,q}(A,W)\lra E_1^{p,q+s}(A,W)\quad\text{for }\ell=2,\]
and
\[P^s_v:E_1^{p,q}(A,W)\lra E_1^{p,q+2s(\ell-1)}(A,W)\quad\text{for }\ell>2.\]
These are also
compatible with $d_1$. We will see in Section \ref{SecWeight} a concrete and accessible description of these vertical operations, for the weight spectral sequence of algebraic varieties. 
Moreover, the operations on $E_1$ of Lemma \ref{myoperations}, are shown to induce operations at each stage $E_r$ with $r>1$.
\end{remark}

\section{Weight filtration}\label{SecWeight}
The weight filtration on the cohomology $H^*(X,\QQ)$ with rational coefficients of any complex algebraic variety $X$ was first introduced by Deligne \cite{DeHII, DeHIII}. The weight filtration was later extended to a filtration on $H^*(X,R)$ where $R$ is an arbitrary commutative ring by Gillet and Soulé (\cite{GilletSoule}, see also \cite{GN}). 
We will refer to this filtration as the \textit{classical weight filtration}. In this section, we define a weight filtration on the complex of singular cochains of any complex algeraic variety, compatible with the $E_\infty$-algebra structure, and which recovers the classical weight filtration on cohomology. We will use the extension criterion of functors of \cite{GN}, which is rests on cohomological descent structures.

\subsection{Cohomological descent structures on $\Ee$-algebras}
A \textit{cohomological descent structure} on a category $\Dd$ is essentially given by a saturated class of \textit{weak equivalences} and a \textit{simple functor} sending every cubical codiagram of $\Dd$ (as a cubical replacement for cosimplicial objects in $\Dd$) to an object of $\Dd$ and
satisfying certain axioms analogous to those of the total complex of a double complex.

We first recall the main ideas on cubical codiagrams and
cohomological descent categories.
We refer to \cite{GN} for the precise definitions.

\begin{definition}
Given a non-empty finite set $S$, denote by $\square_S$ the set of non-empty parts of $S$, ordered by the inclusion. 
We will denote by the same symbol its associated category.
We define a category $\Pi$ as follows. Its objects are given by 
families $S=(S_i)_{i\in I}$ of non-empty finite sets $S_i$, indexed by a variable finite set $I$. 
Given such a family, we will consider the cartesian product 
$\square_S:=\prod_{i\in I} \square_{S_i}$.
Given objects $S=(S_i)_{i\in I}$ and $T=(T_j)_{j\in J}$ of $\Pi$, a morphism $S\to T$ is given by an injective map
$\prod_i S_i\to \prod_j T_j$ such that for any $\alpha=(\alpha_i)\in \square_S$, there is $\beta=(\beta_j)\in \square_T$ with $u(\prod_i \alpha_i)=\prod_j \beta_j$.

The assignement $S\mapsto \square_S$ defines a faithful functor from $\Pi$ to the category of categories.
We will identify $\Pi$ with its image by this functor and so we will denote by $\square_S$ the object $S$ of $\Pi$ and by $u$ the map $\square_u$.
\end{definition}

\begin{definition}
Let $\Dd$ be an arbitrary category. A \textit{cubical codiagram} of $\Dd$ is
a pair $(X,\square)$, where $\square$ is an object of $\Pi$ and $X$ is a functor $X:\square\to \Dd$.
Then $X$ is said to be a \textit{$\square$-codiagram of $\Dd$}.
A \textit{morphism} $(X,\square)\to (Y,\square')$ between cubical codiagrams is given by a pair $(a,\delta)$ where
$\delta:\square'\to \square$ is a morphism of $\Pi$ and $a:\delta^*X:=X\circ\delta \to Y$ is a natural transformation.
\end{definition}
For instance, if $S=\{0\}$ then $\square_S$-codiagrams are just objects $X_1$ of $\Dd$. If $S=\{0,1\}$ then $\square_S$-codiagrams
are diagrams in $\Dd$ of the form 
$X_{1,0}\longleftarrow X_{1,1}\longrightarrow X_{0,1}.$
Denote by $Codiag_{\Pi}\Dd$ the category of cubical codiagrams of $\Dd$.

\begin{definition}
A \textit{cohomological descent category} is given by 
a cartesian category $\Dd$ provided with an initial object $1$, together with
a saturated class of morphisms $\Qq$ of $\Dd$ which is stable by products,
called \textit{weak equivalences}, and a covariant functor 
\[\mathbf{s}:Codiag_{\,\Pi}\Dd\lra \Dd,\] called the \textit{simple functor}.
The data $(\Dd,\Qq,\mathbf{s})$ must satisfy the axioms of Definition 1.5.3 of \cite{GN}.
Objects weakly equivalent to the initial object $1$ are called \textit{acyclic}.
\end{definition}

\begin{example}
The category $\ch{R}$ of cochain complexes of $R$-modules makes the primary example of a cohomological descent category,
with the weak equivalences given by quasi-isomorphisms and the simple
functor defined via the total complex. More generally, 
let $\Dd$ be a category with initial object and a simple functor $\mathbf{s}$.
In a large class of examples, a cohomological descent structure on $\Dd$ is given after lifting the class of
quasi-isomorphisms on $\ch{R}$ via a functor $\Psi:\Dd\lra \ch{R}$, provided that $\Psi$ is compatible
with the simple functor (see Proposition 1.5.12 of \cite{GN}).
For instance, the category $\mathrm{CDGA}_\kk$ of commutative dg-algebras over a field
$\kk$ of characteristic 0,
with the Thom-Whitney simple $\mathbf{s}_{\mathrm{TW}}$ of \cite{Na}, inherits a cohomological descent structure via
the forgetful functor  $\mathrm{CDGA}_\kk\lra \ch{\kk}$.
\end{example}

We next show that, in its cubical version, the normalization functor defines
a cohomological descent structure on the category of $\Ee$-algebras,
with quasi-isomorphisms as weak equivalences.

\begin{definition}\label{normal_square}
Let $A^{\bullet}$ be a $\square$-codiagram of cochain complexes.
Define the \textit{normalization} $N(A^{\bullet})$ of $A^{\bullet}$
as the end
$$N(A^{\bullet}):=\int_\alpha \Hom(C_*(\Delta^\alpha), A^{\alpha})$$
of the functor $\square^{op}_\alpha\times \square_\beta\lra \ch{R}$
given by $(\alpha,\beta)\mapsto \Hom(C_*(\Delta^\alpha), A^{\beta})$.
\end{definition}

The same proof of Theorem \ref{algstrN} for cosimplicial $\Oo$-algebras gives:

\begin{lemma}\label{cubicalnorm}
Let $\Oo$ be an operad.
If  $A^{\bullet}$ is a $\square$-codiagram of $\Oo$-algebras, then 
$N(A^{\bullet})$ is a $(\Oo\otimes\Zz)$-algebra, naturally in $\Oo$ and $A^\bullet$.
\end{lemma}

As in the cosimplicial setting, the above lemma gives a 
normalization functor $$N_\Ee:Codiag_\Pi\Einf_R\lra \Einf_R$$ 
by composing $N$ with the reciprocal image functor of the composition
\[\Ee\lra \Ee\otimes\Ee\lra \Ee\otimes\Zz.\]
\begin{theorem}\label{descent1}
The triple $(\Einf_R,N_\Ee,\Qq_\Ee)$ is a cohomological descent category.
\end{theorem}
\begin{proof}
Denote by $U:\Einf_R\to \ch{R}$ the forgetful functor. By definition, 
the class $\Qq$ is given by the inverse image of quasi-isomorphisms of cochain complexes by $U$.
Also, since restriction morphisms are the identity on the underlying cochain complexes, we have 
$U\circ N_\Ee=N\circ U.$
This proves that we are in the hypotheses of \cite[Proposition 1.5.12]{GN}
on the transfer of cohomological descent structures. Therefore it suffices to prove that 
the triple $(\ch{R},N,\Qq)$ is a cohomological descent category.
This follows from \cite[1.7.2]{GN}, by noting that the simple functor
considered in loc. cit, is isomorphic to $N$. We prove this last claim.

Denote by $C_*(\square_\alpha)$ the chain complex of oriented chains of
the simplicial complex of the non-empty parts of $\square_\alpha$.
Let $C^*(\square_\alpha)$ denote its dual cochain complex, given in each degree by 
$$C^n(\square_\alpha)=\Hom_{R}(C_*(\square_\alpha),R).$$
Given a $\square$-codiagram of cochain complexes $A^{\bullet}$, then its simple is defined by the end 
of the functor $$\square^{op}_\alpha\times \square_\beta\lra \ch{R}
\text{ given by }
(\alpha,\beta)\mapsto C^*(\square_\alpha)\otimes A^{\beta}$$ (see 1.3.1 of \cite{GN}).
Now, we have an isomorphism of cochain complexes
$$C^*(\square_S)\otimes A\cong \Hom(C_*(\square_S),A).$$
Therefore it suffices to prove that
$C_*(\square_\alpha)\cong C_*(\Delta^\alpha)$. This follows from the definition of the oriented chains, since
 $C_p(\square_n)$ is the free $R$-module generated by the non-degenerate $p$-simplices of $\Delta^n$.
\end{proof}

To define a weight filtration on singular cochains we will use a filtered version of the above result.

\begin{theorem}\label{cohdesccats}
Let $F$ be one of the filtrations $t$ (trivial) or $\sigma$ (b\^{e}te).
We have normalization functors $N^F_\Ee$
for filtered $(\Ee,F)$-algebras and
the triples $(\Ff\Einf^t_R,N_\Ee^t,\Qq_0)$ and $(\Ff\Einf^\sigma_R,N_\Ee^\sigma,\Qq_1)$ are cohomological descent categories.
\end{theorem}
\begin{proof}By Lemma \ref{filnorm} together with the results of Section \ref{filnorm} we have filtered normalization functors also in the cubical setting.
The verification  of the cohomological descent structures is parallel to that of Theorem \ref{descent1}, via
the forgetful functor from filtered $ºEe$-algebras to the category of filtered complexes $\ch{\Ff R}$ and
using the fact that both triples $(\ch{\Ff R},N^t,\Qq_0)$ and $(\ch{\Ff R},N^\sigma,\Qq_1)$
are cohomological descent categories (see \cite[Theorem 2.8]{CG3}).
\end{proof}

\subsection{The extension criterion of functors}
The basic idea of the extension criterion of functors is that, given a functor compatible with smooth
blow-ups from the category of smooth schemes to a cohomological descent category, there
exists an extension to all schemes. Such an extension is essentially unique
and is compatible with general blow-ups. Let us review the main statements.

Denote
by $\Sch{\kk}$ the category of reduced schemes, that are separated and of finite type over a field $\kk$
of characteristic 0. Denote by $\Sm{\kk}$ the full subcategory of smooth schemes.

\begin{definition}\label{defelement1_alg}
A cartesian diagram of $\Sch{\kk}$
$$
\xymatrix{
\ar[d]_g\wt Y\ar[r]^j&\wt X\ar[d]^f\\
Y\ar[r]^i&X
}
$$
is said to be an \textit{acyclic square} if $i$ is a closed immersion, $f$ is proper and it induces an isomorphism
$\wt X-\wt Y\to X-Y$.
It is an \textit{elementary acyclic square} if, in addition, 
all the objects in the diagram are in $\Sm{\kk}$, and $f$ is the blow-up of $X$ along $Y$.
\end{definition}

Denote by $\Vv_{\kk}^2$ the category of \textit{good compactifications}: its objects are pairs $(X,U)$ where 
$X$ is a smooth projective scheme and $U$ 
is a smooth open subvariety of $X$ such that  
$D=X-U$ is a normal crossings divisor.
Morphisms are pairs $(f,f'):(X,U)\to (X',U')$ such that the following diagram commutes
$$
\xymatrix{
U\ar[r]\ar[d]_{f'}&X\ar[d]^f\\
U'\ar[r]&X'
}.
$$

\begin{definition}\label{defelement2}
A commutative diagram of $\Vv_{\kk}^2$
$$
\xymatrix{
\ar[d]_g(\wt Y,\wt U\cap \wt Y)\ar[r]^j&(\wt X,\wt U)\ar[d]^{f}\\
(Y,U\cap Y)\ar[r]^i&(X,U)
}
$$
is said to be an \textit{acyclic square} if $f:\wt X\to X$ is proper, $i:Y\to X$ is a closed immersion,
the diagram of the first components is cartesian,
$f^{-1}(U)=\wt U$ and the diagram of the second components is an acyclic square of $\Sch{\kk}$.
\end{definition}

\begin{definition}
A morphism $f:(\wt X,\wt U)\to (X,U)$ in $\Vv_{\kk}^2$ is called \textit{proper elementary modification} if
$f:\wt X\to X$ is the blow-up of $X$ along a smooth center $Y$ which has normal crossings with the
 complement of $U$ in $X$, and if $\wt U=f^{-1}(U)$.
\end{definition}

\begin{definition}
An acyclic square of objects of $\Vv_{\kk}^2$ is said to be an \textit{elementary acyclic square} 
if the map $f:(\wt X,\wt U)\to (X,U)$ is a proper elementary modification,
and the diagram of the second components is an elementary acyclic square of $\Sm{\kk}$.
\end{definition}

\begin{definition}[$\Phi$-rectified functors]
Let $\Dd$ be a cohomological descent category and let 
$\square\in \Pi$. Denote by $\Dd^\square:=\mathrm{Fun}(\square,\Dd)$ the category of diagrams of type $\square$ in $\Dd$.
The simple functor induces a functor $\Ho(\Dd^\square)\to \Ho(\Dd)$.
In general, we are interested in cubical diagrams in $\Ho(\Dd)$
and we do not have a simple functor $\Ho(\Dd)^\square\to \Ho(\Dd)$.
The notion of \textit{$\Phi$-rectified functor} corresponds, roughly speaking, to functors
$F:\Cc\to \Ho(\Dd)$ which are defined on all cubical diagrams in the form
$F^\square:\Cc^\square\to \Ho(\Dd^\square)$,
so that we can take the composition $\Cc^\square\to \Ho(\Dd^\square)\to \Ho(\Dd)$
(see 1.6 of \cite{GN}).
\end{definition}

The following is a relative version of the main result of \cite{GN} on the 
extension of functors from schemes with values in cohomological descent categories.
\begin{theorem}[\cite{GN}, Theorem 2.3.6] \label{extensiorel}Let $\Dd$ be a cohomological descent category with simple functor $\mathbf{s}$
and let
$$F:\Vv_{\kk}^2\lra \Ho(\Dd)$$ be
a contravariant $\Phi$-rectified functor satisfying:
\begin{enumerate}
 \item [{(F1)}] $F(\emptyset,\emptyset)$ is the final object of $\Dd$ and $F((X,U)\sqcup (Y,V))\to F(X,U)\times F(Y,V)$ is an isomorphism.
 \item [{(F2)}] If $(X_\bullet,U_\bullet)$ is an elementary acyclic square of $\Vv_{\kk}^2$, then
$\mathbf{s}F(X_\bullet,U_\bullet)$ is acyclic.
\end{enumerate}
Then there exists a contravariant $\Phi$-rectified
functor 
$$F':\Sch{\kk}\lra \Ho(\Dd)$$
such that:
\begin{enumerate}
\item [(1)]If $(X,U)$ is an object of $\Vv_{\kk}^2$, then $F'(U_\infty)\cong F(X,U)$.
 \item [(2)]If $U_{\bullet}$ is an acyclic square of $\Sch{\kk}$, then $\mathbf{s}F'(U_{\bullet})$ is acyclic.
 \end{enumerate}
In addition, the functor $F'$ is essentially unique.
\end{theorem}

\subsection{Weight filtration on the complex of singular cochains}
Let $X$ be a topological space. Both the Godement cosimplicial resolution functor and the global
sections functor define functors of sheaves of filtered algebras. Also, as in 
Theorem \ref{cohdesccats}, the trivial filtration on $C_*(\Delta^\bullet)$ gives a filtered cosimplicial
normalization functor $N^t_{\Ee}$.
This gives a \textit{derived global sections functor} for filtered $\Ee$-algebras 
\[\RR^t_{\Ee}\Gamma(X,-):=N^t_{\Ee}\circ \Gamma(X,-)\circ G^+.\]

Let $(X,U)$ be an object of $\Vv_\CC^2$ and denote by $j:U\hookrightarrow X$ the inclusion.
Consider the sheaf $\RR_\Ee j_*\underline{R}_U$ and endow it with the canonical filtration $\tau$.
This is a sheaf of filtered $(\Ee,\tau)$-algebras but, by Lemma \ref{maptosigma}, we can consider it as a
sheaf of filtered $(\Ee,t)$-algebras.

\begin{prop}
The assignment $(X,U)\mapsto \RR^t_{\Ee}\Gamma(X,(\RR_{\Ee}j_*\underline{R}_U,\tau))$
defines a functor 
$$\Ww_R':\Vv_\CC^2\lra \Ff\Einf_R^t$$
satisfying the following conditions:
\begin{enumerate}
 \item There is a string of quasi-isomorphisms of $\Ee$-algebras from $\Ww_R'(X,U)$ to  $C^*(U,R)$.
 \item The filtration induced on $H^*(U,R)\cong H^*(\Ww_R'(X,U))$ is the classical weight filtration.
\end{enumerate}
\end{prop}

\begin{proof}The functoriality of $(X,U)\mapsto \RR^t_{\Ee}\Gamma(X,(\RR_{\Ee}j_*\underline{R}_U,\tau))$ follows as in Theorem \ref{quistocochains_thm}.
Let us prove (1). By Lemma  \ref{relativesheaf}, the $\Ee$-algebras $\RR_{\Ee}\Gamma(X,\RR_{\Ee}j_*\underline{R}_U)$ and 
$\RR_{\Ee}\Gamma(U,\underline{R}_U)$ are naturally quasi-isomorphic. The result now follows from Theorem \ref{quistocochains_thm}.
In particular, we have an isomorphism 
$$H^*(U,R)\cong H^*(\Ww_R'(X,U)).$$
To prove (2), it suffices to note that, when forgetting the $\Ee$-algebra structures,
the definition of $\Ww_R'(X,U)$ coincides with Deligne's construction of the 
weight filtration for smooth compactifications (see \cite{DeHII} for rational coefficients and \cite{GilletSoule}, \cite{GN} for the arbitrary ground ring case).
\end{proof}

We now use the extension criterion of functors of Theorem \ref{extensiorel} to extend the functor $\Ww_R'$ to the category of all schemes.
Recall that $\Ff\Einf_R^\sigma$ denotes the category of filtered $(\Ee,\sigma)$-algebras and 
$\Ho(\Ff\Einf_R^\sigma)$ is its localization with respect to $E_r$-quasi-isomorphisms.
We get:

\begin{theorem}\label{mainteoweight}
 There exists a $\Phi$-rectified functor \[\Ww_R:\Sch{\CC}\lra \Ho(\Ff\Einf_R^\sigma)\] such that:
\begin{enumerate}[(1)]
\item There is a quasi-isomorphism  of $\Ee$-algebras $\Ww_R(X)\simeq C^*(X,R)$ for every $X\in \Sch{\CC}$.
\item If $(X,U)$ is a pair in $\Vv_{\CC}^2$ then $\Ww_R(U)\cong \Ww_R'(X,U)$.
 \item For every $p,q\in\ZZ$ and every acyclic square of $\Sch{\CC}$
$$
\xymatrix{
\ar[d]\wt Y\ar[r]&\wt X\ar[d]\\
Y\ar[r]&X
}
$$
there is a long exact sequence of $R$-modules
\begin{align*}
\cdots\to E_2^{p,q}(\Ww_R(X))\to
E_2^{p,q}(\Ww_R(\wt X))\oplus E_2^{p,q}(\Ww_R(Y)) \to \\\to
E_2^{p,q}(\Ww_R(\wt Y))\to E_2^{p+1,q}(\Ww_R(X))\to \cdots
\end{align*}

\item The filtration induced by $\Ww_R(X)$
on $H^n(X,R)$ is the classical weight filtration.
\end{enumerate}
\end{theorem}

\begin{proof}Since $\Ww_R'$ has values in $\Ff\Einf_R^t$, it automatically gives a $\Phi$-rectified functor
\[\Ww_R':\Vv_\CC^2\lra \Ff\Einf_R^t\lra \Ff\Einf_R^\sigma\lra \Ho(\Ff\Einf_R^\sigma)\]
when composing with the reciprocal image functor of the map $(\Ee,\sigma)\to (\Ee,t)$ and the localization functor at $E_1$-quasi-isomorphisms.
By Theorem \ref{cohdesccats} the triple $(\Ff\Einf_R^\sigma,N_\Ee^\sigma,\Qq_1)$ is a  cohomological descent category.
Therefore by Theorem \ref{extensiorel} it suffices to show that $\Ww_R'$
satisfies properties (F1) and (F2).
Property (F1) is trivial. 
Condition (F2)
is equivalent to the condition that
the map
$${\Ww_R'}(X,U)\lra N_\Ee^\sigma(\Ww_R'(X_\bullet,U_\bullet))$$ is an $E_1$-quasi-isomorphism
for every elementary acyclic square  
$(X_\bullet,U_\bullet)\to (X,U)$ of $\Vv_\CC^2$.
This is a statement at the level of the underlying cochain complexes, and hence it suffices to show that
the map
$${\Ww_R'}(X,U)\lra \mathbf{s}^\sigma(\Ww_R'(X_\bullet,U_\bullet))$$ is an $E_1$-quasi-isomorphism,
where now 
$$\Ww_R'(X,U)=\RR\Gamma(X,(\RR j_*\underline{R}_U,\tau))$$
is defined via the derived additive functors $\RR j_*$ and $\RR\Gamma(X,-)$,
and $\mathbf{s}^\sigma$ is the total simple functor that sends a codiagram of filtered complexes
complexes $(K^\bullet,W^\bullet)$ to the filtered complex 
$$W_p(\mathbf{s}^\sigma(K^\bullet,W^\bullet))=
\bigoplus_{|\alpha|=0} W_pK^\alpha\oplus \bigoplus_{|\alpha|=1} W_{p+1}K^\alpha[-1]\oplus\cdots \oplus \bigoplus_{|\alpha|=r} W_{p+r}K^\alpha[-r] \oplus\cdots.
$$
The result now follows as in the definition of the weight filtration in cohomology
(see Theorem 2.17 of \cite{CG3}). We sketch the main steps for completeness.

For any codiagram of filtered complexes $K^\bullet$, we have that (see Proposition 2.7 of \cite{CG3})
$$E_1^{*,q}(\mathbf{s}^\sigma(K^\bullet))\stackrel{\sim}{\longleftrightarrow}\mathbf{s}E_1^{*,q}(K^\bullet).$$
Therefore
the above reduces to prove that 
for all $q\in\ZZ$, the sequence 
\begin{align*}\cdots 
\to E_2^{*,q}(\Ww_R'(X,U))\to
E_2^{*,q}(\Ww_R'(\wt X,\wt U))\oplus E_2^{*,q}(\Ww_R'(Y, U\cap Y)) \to\\\to E_2^{*,q}(\Ww_R'(\wt Y, \wt U\cap \wt Y))\to E_2^{*+1,q}(\Ww_R'(X,U))\to\cdots
\end{align*}
is exact, where 
$$
\xymatrix{
\ar[d]_g(\wt Y,\wt U\cap \wt Y)\ar[r]^j&(\wt X,\wt U)\ar[d]^{f}\\
(Y,U\cap Y)\ar[r]^i&(X,U)
}
$$
is an elementary acyclic square. We may identify
$E_1^{*,q}(\Ww_R'(X,U))$ with the Gysin complex
$G^q(X,U)^*$, given in each degree by 
$$G^q(X,U)^p:=H^{q+2p}(D^{(-p)},R)\text{, where } D=X-U$$
(see \cite{DeHII}, 3.1.9 and 3.2.4, see also Section 4.3 of \cite{PS}).
Now, by Proposition 4.5 of \cite{CG3}, we have:
 \begin{enumerate}
 \item If $Y\nsubseteq D$ then the simple of the double complex
 \[0\to G^q(X,U)\to G^q(\wt X,\wt U)\oplus G^q(Y,U\cap Y)\to G^q(\wt Y,\wt U\cap \wt Y)\to 0
\]
 is acyclic for all $q$

 \item If $Y\subset D$ then $G^q(X,U)\to G^q(\wt X,\wt U)$ is a quasi-isomorphism for all $q$.\qedhere
\end{enumerate}
\end{proof}

\begin{remark}\label{alsovalid}After minor modifications, the above theorem is also valid in the following contexts:
\begin{enumerate}[(1)]
\item \textit{Real algebraic varieties.} By restricting to $\FF_2$-coefficients, for which  
Poincar\'{e}-Verdier duality holds, one may define a functor
$\Ww_{\FF_2}:\Sch{\RR}\lra \Ho(\Ff\Einf_{\FF_2}^\sigma)$.
This result promotes the results of MacCrory and Parusi{\'n}ski \cite{MCPI}, \cite{MCPII}
on the existence of a weight filtration for the $\FF_2$-homology of real algebraic varieties,
from the additive, to the $E_\infty$-setting.
\item \textit{Analytic spaces.} A version of the extension criterion of functors for analytic spaces (see Theorem 3.9 of \cite{CG3}),
gives a functor $\Ww_R:\mathrm{An}_\CC^\infty\lra \Ho(\Ff\Einf_R^\sigma)$, where $\mathrm{An}_\CC^\infty$ denotes the category of complex analytic spaces
endowed with an equivalence class of compactifications. 
The same is true for real analytic spaces and $\FF_2$-coefficients. Therefore we also obtain a multiplicative weight filtration in these settings.

\item \textit{Rational homotopy.}
The theory of mixed Hodge structures in rational homotopy 
gives a weight filtration on 
Sullivan's algebra $\Aa_{PL}(X)$ of piece-wise linear forms.
Also, Theorem \ref{mainteoweight} gives a filtered $E_\infty$-algebra $\Ww_\QQ(X)$
over $\QQ$. These two constructions 
are naturally weakly equivalent in the category of filtered $(\Ee,\sigma)$-algebras.

\item \textit{Rational homotopy of analytic spaces.}
By the previous remark, we also obtain a weight filtration on the rational homotopy type of every complex analytic space endowed with an equivalence class of compactifications, extending the theory of mixed Hodge structures on the rational homotopy type of complex algebraic varieties, to complex analytic spaces. In the analytic setting, 
the weight filtration is not part of a mixed Hodge structure, but it is a new and well-defined invariant.
\end{enumerate}
\end{remark}

\subsection{Steenrod operations on the weight spectral sequence}
The existence of a weight filtration compatible with the $E_\infty$-structure over $\FF_\ell$ gives a whole new family of invariants,
given by Steenrod (for $\ell=2)$ and reduced $\ell$-th power operations (for $\ell>2$) on the successive pages of the weight spectral sequence. 
These operations may be useful in order to distinguish homotopy types of algebraic varieties.

\begin{definition}\label{multiweightdef}
Let $X$ be a complex algebraic variety.
The \textit{weight spectral sequence of $X$} with coefficients in a commutative ring $R$
is the spectral sequence 
associated with the filtered $(\Ee,\sigma)$-algebra $\Ww_R(X)$ given in Theorem \ref{mainteoweight}:
$$E_1^{-p,q}(X,R):=H^{q-p}(Gr_{p} \Ww_R(X))\Longrightarrow H^{q-p}(X,R).$$
\end{definition}

\begin{remark}
As shown by Deligne, for $R=\QQ$ the weight spectral sequence degenerates at the $E_2$ page, but this is not the case for a general coefficient ring (see \cite{GilletSoule}).
However, for $r\geq 2$, the pair
$(E_r(X,R),d_r)$ is a well-defined algebraic invariant of $X$, which differs in general from the 
associated graded groups to the weight filtration in cohomology.
\end{remark}

In the general quasi-projective singular case, the weight spectral sequence and its multiplicative structure
is hard to describe, as it is given by the normalization of Lemma \ref{E1lemma}. In the following particular cases it becomes much more accessible.

\subsubsection*{Smooth varieties}
Given a smooth complex variety $U$ consider a smooth compactification $U\hookrightarrow X$ in such a way that the complement $D:=X-U=D_1\cup\cdots \cup D_N$ is  the union of
of irreducible smooth varieties $D_i$ meeting transversally.
Let $D^{(0)}=X$ and for all $p>0$, denote by $D^{(p)}=\bigsqcup_{|I|=p}D_I$
the disjoint union of all 
$p$-fold intersections 
$D_I:=D_{i_1}\cap\cdots \cap D_{i_p}$ where $I=\{i_1,\cdots,i_p\}$ denotes an ordered subset of $\{1,\cdots,N\}$.
Since $D$ has normal crossings, it follows that $D^{(p)}$ is a smooth projective variety of dimension $n-p$. 
For $1\leq k\leq p$, denote by $j_{I,k}:D_I\hookrightarrow D_{I\setminus \{i_k\}}$ the inclusion 
and let \[j_{p,k}:=\bigoplus_{|I|=p} j_{I,k}:D^{(p)}\hookrightarrow D^{(p-1)}.\]
This defines a simplicial resolution
$D_\bullet=\{D^{(p)}, j_{p,k}\}\lra D$, called the \textit{canonical hyperresolution} of $D$.
The weight spectral sequence of Definition \ref{multiweightdef}, with coefficients in an arbitrary commutative ring $R$ is then given by the relative cohomology groups
\[E_1^{-p,q}(U)=H^q(X,X- D^{(p)}, R).\]
The differential 
 \[d_1:E_1^{-p,q}(U)=H^{*}(X,X-D^{(p)})\lra E_1^{-p+1,q}(U)=H^*(X,X-D^{(p-1)})\] is 
naturally induced by the inclusion $X-D^{(p-1)}\hookrightarrow X-D^{(p)}$
as a combinatorial sum 
$$\sum_{k=1}^p (-1)^{k} j_{p,k}^*:H^{*}(X,X-D^{(p)})\lra H^*(X,X-D^{(p-1)}).$$
induced by the inclusions $j_{p,k}$.
 
\begin{remark}In the additive setting, 
Deligne's weight spectral sequence is usually written as 
\[E_1^{-p,q}(U)=H^{q-2p}(D^{(p)})\]
and the differential is the combinatorial sum of Gysin maps. 
This uses the Thom isomorphism $H^{*-2p}(D^{(p)})\cong H^*(X,X-D^{(p)})$.
Note that Steenrod operations are in general not compatible with the Thom isomorphism and so this identification should not be done in the multiplicative setting of $E_\infty$-algebras.
It does work, however, when studying the multiplicative weight spectral sequence with coefficients in $\QQ$, for which there is a commutative algebra structure (see \cite{Mo}) as well as in some particular cases,
as we will later see.
\end{remark}

The vertical operations  
\[P^s_v:E_1^{-p,q}(U)\lra E_1^{-p,q+s}(U)\quad,\text{ for }R=\FF_2\]
and 
\[P^s_v:E_1^{-p,q}(U)\lra E_1^{p,q+2s(\ell-1)}(U)\quad,\text{ for }R=\FF_\ell\text{ with }\ell>2\]
are induced by the corresponding operations $H^*(X,X-D^{(p)})$ on relative cohomology.
These are compatible with the maps \[H^{*}(X,X-D^{(p)})\lra H^*(X,X-D^{(p-1)})\]
and so are compatible with $d_1$. Therefore the above operations give well-defined operations in $E_2$. In addition, Lemma \ref{myoperations} gives the corresponding non-vertical operations on $E_2$.

By Remark \ref{alsovalid} the above discussion is also valid in the case of real algebraic varieties and coefficients in $\FF_2$, for which it is easy to compute multiple interesting examples.

\begin{example}
Consider the quadric hypersurface $\mathbf{Q}$ of $\RR\PP^7$ given by
\[\mathbf{Q}=\{(x_0:\cdots:x_7)\in \RR\PP^7; x_0^2+x_1^2+x_2^2+x_3^2-x_4^2-x_5^2-x_6^2-x_7^2=0\}.\]
Let $U=\RR\PP^7-\mathbf{Q}$. 
The cohomology of $\mathbf{Q}$ is isomorphic to
$H^*(\RR\PP^3,R)\otimes_R H^*(S^3,R)$. This may be checked using the Gysin sequence for the sphere bundle
$\mathbf{Q}\to \RR\PP^3$ defined by projecting to the last components
$(x_0:\cdots:x_7)\mapsto (x_4:\cdots:x_7)$.
In particular we may write
\[H^*(\mathbf{Q},\FF_2)\cong \FF_2[t,s]/(t^4,s^2), \text{ with }|t|=1, |s|=3.\]
Cartan's formula allows us to determine the Steenrod operations on $H^*(\mathbf{Q},\FF_2)$:
\[Sq^s(t^is^j)=Sq^s(t^i)s^j.\]
Therefore for $s>0$ the only non-trivial operations on $H^*(\mathbf{Q},\FF_2)$ are 
\[Sq^1(t)=t^2\text{ and }Sq^1(ts)=t^2s.\]
To compute the weight spectral sequence of $U$ we still need to compute the relative cohomolohy $H^*(\RR\PP^7,\RR\PP^7-\mathbf{Q})$. By the Thom isomorphism we have an isomorphism
\[H^{*-1}(\mathbf{Q})\cong H^*(\RR\PP^7,\RR\PP^7-\mathbf{Q})\]
as modules over the Steenrod algebra. Indeed, while the Thom isomorphism does not commute with Steenrod operations for an arbitrary vector bundle,
since the normal bundle to the embedding 
 $\mathbf{Q}\to \RR\PP^7$ is trivial, the obstructions to commutativity given by the Stiefel-Whitney classes of this bundle, vanish for $i>0$.
 Writing 
 \[H^*(\RR\PP^7,\FF_2)\cong \FF_2[x]/(x^8)\] 
We obtain the following vertical Steenrod operations in $E_1$:

\[
\xymatrix@R=12pt@C=36pt{
&\tau t^3s&x^7\\
&\tau t^2s&x^6\ar[u]^-{Sq^1}\\
&\tau ts \ar[u]^-{Sq^1}&x^5\\
E_1^{*,*}(U,\FF_2)\cong &\tau s, \tau t^3&x^4\\
&\tau t^2&x^3\ar@/_1pc/[uu]_{Sq^2}\ar@/_3pc/[uuu]_{Sq^3}\ar[u]^-{Sq^1}\\
&\tau t \ar[u]^-{Sq^1}&x^2\ar@/_1pc/[uu]_{Sq^2}\\
&\tau&x\ar[u]_-{Sq^1}\\
&0&1
}
\]
Now, the differential $d_1$ of the weight spectral sequence is the Gysin map 
 \[i_!:H^{*-1}(\mathbf{Q})\cong \FF_2[t,s]/(t^4,s^2)\lra H^*(\RR\PP^7)\cong \FF_2[x]/(x^8),\]
 given by $\tau t^i\mapsto 0$ and $\tau t^i s\mapsto x^{4+i}$.
We get the following table for $E_2\cong H^*(U)$:
\[
\xymatrix@R=12pt@C=36pt{
&\tau t^3&0\\
&\tau t^2&x^3\\
E_2^{*,*}(U,\FF_2)\cong&\tau t \ar[u]^-{Sq^1}&x^2\\
&\tau&x\ar[u]_-{Sq^1}\\
&0&1
}
\]
In particular, we read non-trivial Steenrod operations $Sq^1$ on both columns of $E_2$.
\end{example}

\subsubsection*{Singular projective varieties}
Let $X$ be a singular complex projective variety of dimension $n$ and let $X_\bullet\to X$ be a cubical hyperresolution of $X$,
where $X_\alpha$ is a smooth projective variety of dimension $\dim (X_\alpha)\leq n-|\alpha|+1$. As shown in \cite{GNPP},
the classical weight spectral sequence for $X$ with coefficients in $\FF_\ell$ is given by
\[
E_1^{p,q}(X,\FF_\ell)=\bigoplus_{|\alpha|=p} H^q(X_\alpha,\FF_\ell)\]
The differential $d_1:E_1^{p,q}(X,\FF_\ell)\lra E_1^{p+1,q}(X,\FF_\ell)$
is given by the combinatorial sum of restriction morphisms $H^q(X_p,\FF_\ell)\to H^q(X_{p+1},\FF_\ell)$.

In contrast, the weight spectral sequence of Definition \ref{multiweightdef} is given in each bidegree by 
\[
E_1^{-p,q}(X,\FF_\ell)=\int_\alpha \Hom(C_p(\Delta^\alpha),H^q(X_\alpha)),
\]
where we used Lemma \ref{E1lemma}. The operations of Lemma \ref{myoperations} and Remark 
\ref{verticals} are available. In particular, the horizontal operations 
\[Sq^s:E_2^{p,0}(X,\FF_\ell)\lra E_2^{p+s,0}(X,\FF_\ell)\text{ for }\ell=2\]
and 
\[P^s:E_2^{p,0}(X,\FF_\ell)\lra E_2^{p+2s(\ell-1),0}(X,\FF_\ell)\text{ for }\ell>2\]
prove to be powerful invariants in order to distinguish homotopy types of algebraic varieties, as we next show.

\begin{example}[Isolated singularities and the dual complex]
Let $X$ be a projective variety with only isolated singularities.
Then for all $p\geq 0$ we have an isomorphism 
\[E_2^{p,0}(X,\FF_2)\cong H^p(\Sigma D_X,\FF_2)\]
where $D_X$ denotes the dual complex associated to $X$, whose homotopy type is an invariant of the variety
(see \cite{Payne},\cite{ABW}) and $\Sigma D_X$ denotes its suspension.
In addition, by a result of Koll\'{a}r \cite{Kollar}, for any finite simplicial complex $V$ there is 
a projective variety $X$ with only isolated singularities 
such that $D_X$ is homotopy equivalent to $V$.
This allows one to produce examples of algebraic varieties that may 
be distinguished by the horizontal operations in $E_2^{*,0}$.
\end{example}

\section{Intersection cohomology}\label{SecInter}

Intersection cohomology is a Poincar\'{e} duality cohomology theory for singular spaces.
It is defined for any topological pseudomanifold and
depends on the choice of a multi-index called \textit{perversity}, measuring how far cycles are
allowed to deviate from transversality. 
Goresky and MacPherson, first defined intersection homology in \cite{GMP1} as the homology
$IH^{\overline p}_*(X)$ of a subcomplex $\mathrm{IC}_*^{\overline{p}}(X)$ of the ordinary chains of $X$,
given by those cycles which meet the singular locus of $X$ with a
controlled defect of transversality with respect to the chosen perversity $\ov p$ and stratification.
Subsequently, Deligne proposed a sheaf-theoretic approach, which was developed in
\cite{GMP2}. In this section, we promote Deligne's additive sheaf to a sheaf carrying an $\Ee$-structure.

\subsection{Algebras with perversities}
The intersection cohomology of a topological pseudomanifold does not define an algebra, but 
it has a product that is compatible
with perversities. This was formalized by Hovey \cite{Hov2}, 
who defined a model structure on the category of perverse differential graded algebras. 
We next review this formalism.

Throughout this section we let $n\geq 0$ be a fixed integer. This integer 
will later be prescribed by a fixed stratification on a topological pseudomanifold.

\begin{definition}
A \textit{perversity} is a sequence of integers $\ov{p}=(p(2),\cdots,p(n))$ 
satisfying the conditions $p(2)=0$ and $p(k)\leq p(k+1)\leq p(k)+1$.
\end{definition}
Denote by $\Pp$ the poset of perversities, where we say that 
$\ov p\leq \ov q$ if and only if $p(k)\leq q(k)$ for all $k$.
Any perversity lies between the trivial perversity $\ov 0=(0,\cdots,0)$
and the top perversity $\ov t=(2,3,4,\cdots)$.

In order to consider multiplicative structures, we want to define a monoidal 
structure on the poset of perversities. For this purpose, we consider an enlarged set 
by letting \[\widehat{\Pp}:=\Pp\cup\{\ov{\infty}\}\]
with the property $\ov p< \ov \infty$ for all $\ov p\in \Pp$.

\begin{definition}
Define a bifunctor $\oplus:\widehat {\Pp}\times \widehat{\Pp}\to \widehat{\Pp}$ by letting 
$\ov p\oplus \ov q$ to be the smallest element $\ov r$ in $\widehat{\Pp}$ such
that $p(k)+q(k)\leq r(k)$ for all $k$. 
\end{definition}

This operation allows an inductive description (see \cite{CST}).
Note that if there is some $k$ such that $p(k)+q(k)>k-2$ then we have $\ov p\oplus \ov q=\ov \infty$.
The triple $(\widehat{\Pp},\oplus,\ov 0)$ is a symmetric monoidal category
with the trivial perversity as left and right identity for $\oplus$.

\begin{definition}
A \textit{perverse $\Ee$-algebra} is given by a functor $A_{\ov \bullet}:\widehat{\Pp}\lra \ch{R}$ together 
with structure morphisms
 \[\theta(\ell):\Ee(\ell)\otimes( A_{\ov p_1}\otimes \cdots\otimes A_{\ov {p}_\ell})\lra
 A_{\ov{p}_1\oplus \cdots\oplus \ov{p}_\ell}\]
 for each $\ell$ and each tuple of perversities $\ov p_1,\cdots, \ov p_\ell$
 such that:
 \begin{enumerate}
 \item The morphisms $\theta(\ell)$ are subject
to natural associativity, unit and $\Sigma_\ell$-equivariance constraints.
\item For every two sets of perversities $\ov p_1,\cdots, \ov p_\ell$ and $\ov q_1,\cdots, \ov q_\ell$
with $\ov p_i\leq \ov q_i$, for all $1\leq i\leq \ell$, the following diagram commutes:
\[\xymatrix{
\ar[d]\Ee(\ell)\otimes (A_{\ov p_1}\otimes \cdots\otimes A_{\ov {p}_\ell})\ar[r]^-{\theta(\ell)}& A_{\ov{p}_1\oplus \cdots\oplus\ov{p}_\ell}\ar[d]\\
\Ee(\ell)\otimes (A_{\ov q_1}\otimes \cdots\otimes A_{\ov {q}_\ell})\ar[r]^-{\theta(\ell)}& A_{\ov{q}_1\oplus\cdots\oplus\ov{q}_\ell}
}\]
where the vertical maps are induced by the poset maps.
 \end{enumerate}
\end{definition}

Equivalently, consider the symmetric monoidal structure on 
$\mathbf{Fun}(\widehat{\Pp},\chb{\geq 0}{R})$ given by 
\[(A_{\ov\bullet}\otimes B_{\ov \bullet})_{\ov p}:= \sum_{\ov q\oplus \ov q'=\ov p}A_{\ov q}\otimes B_{\ov q'}.\]
Then a perverse $\Ee$-algebra is just an $\Ee$-algebra in $\mathbf{Fun}(\widehat{\Pp},\chb{\geq 0}{R})$.
We will denote by $\widehat{\Pp}\Einf_R$ the category of perverse $\Ee$-algebras over $R$.

\begin{remark}\label{T0}
The following construction is due to Hovey \cite{Hov2}. 
For each perversity $\ov p$ there is a forgetful functor from perverse complexes to complexes
\[U_{\ov p}:\widehat{\Pp}\ch{R}\lra \ch{R}\]
given by $A_{\ov \bullet}\mapsto A_{\ov p}$.
This functor has an adjoint $T_{\ov p}$ defined by $T_{\ov p}(A):=A$ if $\ov q\leq \ov p$ and zero otherwise.
The functors $U_{\ov 0}$ and  $T_{\ov 0}$ associated to the trivial perversity are monoidal
so we obtain an adjoint pair 
\[U_{\ov 0}:\widehat{\Pp}\Einf_R \rightleftarrows \Einf_R:T_{\ov 0}.\]
\end{remark}

\subsection{Multiplicative intersection complex}
Two main ingredients in the definition of Deligne's intersection complex are truncations of complexes of sheaves and derived direct image functors.
We begin by reviewing these constructions in the setting of sheaves of perverse $\Ee$-algebras.

\begin{definition}
Let $A\in\chb{\geq 0}{R}$ be a cochain complex and $k\in\ZZ_{\geq 0}$ a non-negative integer.
The \textit{$k$-truncation} of $A$ is the complex  $\tau_{\leq k}A$ given by
\[(\tau_{\leq k}A)^m=\left\{
\begin{array}{ll}
0& k>m\\
\Ker(d)\cap A^m& k=m\\
A^m& k<m
\end{array}
\right..\]
\end{definition}
For every $k\geq 0$, this defines a $k$-truncation functor 
\[\tau_{\leq k}:\chb{\geq 0}{R}\lra \chb{\geq 0}{R}\]
satisfying the following properties (see for example Section 1.14 of \cite{GMP2}):
\begin{enumerate}
 \item $\tau_{\leq k}\circ \tau_{\leq k'}A=\tau_{\leq \mathrm{min}(k,k')}A$.
 \item $\tau_{\leq k}A\subseteq  \tau_{\leq k'}A\subseteq A$ for every $k\leq k'$.
 \item If $\varphi:A\to B$ induces isomorphisms $H^i(A)\cong H^i(B)$ for all $i\leq k$ then 
 the $k$-truncated map $\tau_{\leq k}\varphi:\tau_{\leq k}A\to \tau_{\leq k}B$ is a quasi-isomorphism.
 \item If $f:X\to Y$ is a continuous map of topological spaces, and $\Aa$ is a sheaf of complexes,
 then $\tau_{\leq k}f^*\Aa\cong f^*\tau_{\leq k}\Aa$.
\end{enumerate}

The truncation functor $\tau_{\leq k}$ is not a lax monoidal,
but it defines a lax monoidal functor
\[\tau_{\leq \bullet(k)}: \mathbf{Fun}(\widehat{\Pp},\chb{\geq 0}{R})\lra \mathbf{Fun}(\widehat{\Pp},\chb{\geq 0}{R})\]
by letting
\[\left(\ov p\mapsto A_{\ov p}\right)\mapsto \left(\ov p\mapsto \tau_{\leq p(k)}A_{\ov p}\right).\]
Therefore it induces a $k$-truncation functor 
$\tau_{\leq \bullet(k)}$ in the category of perverse $\Ee$-algebras.

Note as well that the constructions of Section 
\ref{sheaveseinfty} naturally extend from $\Ee$-algebras to perverse $\Ee$-algebras.
In particular, we have a $k$-truncation functor on the category of sheaves of perverse $\Ee$-algebras
$\Sh_X(\widehat{\Pp}\Einf_R)$
and for any continuous map
$f:X\to Y$
we have  derived direct image and inverse functors
\[\RR_\Ee f_*:\Sh_X(\widehat{\Pp}\Einf_R)\lra \Sh_Y(\widehat{\Pp}\Einf_R)\text{ and }f^*:\Sh_Y(\widehat{\Pp}\Einf_R)\lra \Sh_X(\widehat{\Pp}\Einf_R).\]
Also, by Lemma \ref{co_unit} there are maps
$1\to \RR_\Ee f_*\circ f^*$ and  $f^*\circ \RR_\Ee f_*\stackrel{\sim}{\leftarrow} \bullet \to 1$
in the category of sheaves of perverse $\Ee$-algebras.

Recall that a
 \textit{topological pseudomanifold of dimension $n$} is a space $X$ admitting a stratification 
\[X=X_n\supset X_{n-2}\supset X_{n-3}\supset\cdots\supset X_1\supset X_0\supset \emptyset\]
such that $X-X_{n-2}$ is an oriented dense $n$-manifold 
and $X_i-X_{i-1}$ is an $i$-manifold along which the normal structure of $X$ is locally trivial.

Via Whitney's theory of stratifications, complex and real algebraic varieties or, more generally, complex and real analytic varieties are examples of topological pseudomanifolds, with $\Sigma=X_{n-2}$ the singular set of the space. We refer to \cite{GMP1} and \cite{GMP2} for details on topological stratifications.

Throughout this section we will let $X$ be a topological pseudomanifold
with singular locus $\Sigma$ and a fixed stratification $\{X_i\}$ with $\Sigma=X_{n-2}$.
Define sets $U_k:=X-X_{n-k}$ and denote by 
\[i_k:U_k\to U_{k+1}\text{ and }j_k:U_{k+1}-U_k\to U_{k+1}.\]
the inclusions. Note that $U_2=X-\Sigma$ is the regular part of $X$ and $U_{n+1}=X$.

\begin{definition}\label{defi_ICE}
The \textit{intersection $\Ee$-algebra $\mathcal{IC}_{\overline{\bullet}}(X,R)$ of $X$} 
is the sheaf of perverse $\Ee$-algebras defined inductively as follows:
let $\mathbb{P}_{\ov\bullet}^{(2)}$ be the sheaf of perverse $\Ee$-algebras
given by the constant sheaf on $X-\Sigma$ in each perversity:
\[\mathbb{P}_{\ov p}^{(2)}:=\underline{R}_{X-\Sigma}\text{ for all }\ov p\in \widehat{\Pp}.\]
 For $k\geq 2$, define
 \[\mathbb{P}_{\ov\bullet}^{(k+1)}:= \tau_{\leq \bullet (k)}\circ \RR_\Ee i_{k*}\mathbb{P}_{\ov\bullet}^{(k)}.\]
We then let 
\[\mathcal{IC}_{\overline{\bullet}}(X,R):=\mathbb{P}_{\ov\bullet}^{(n+1)}.\]
In each perversity, we have:
\[\mathcal{IC}_{\overline{p}}(X,R)= 
\tau_{\leq {p}(n)}\circ \RR i_{n*}\circ 
\cdots\circ 
\tau_{\leq {p}(3)}\circ \RR i_{3*}\circ \tau_{\leq {p}(2)}\circ \RR i_{2*}\underline{R}_{U_2}.\]
\end{definition}
Therefore, by forgetting the $\Ee$-algebra structure and fixing a finite perversity $\ov p$, 
we recover Deligne's intersection complex 
given in \cite{GMP2}, modulo a shift by $n$ in the degrees of the complexes.
In particular, the intersection cohomology of $X$ in degree $i$ and perversity $\ov p$
is given by 
\[IH^i_{\ov p}(X,R)\cong \HH^{i}(\mathcal{IC}_{\overline{p}}(X,R))\cong H^{i}(
\RR_\Ee\Gamma(X,\mathcal{IC}_{\overline{p}}(X,R)). 
\]
For the perversity $\ov\infty$,
by Lemma \ref{compoderived}, the sheaf $\mathcal{IC}_{\overline{\infty}}(X,R)$ is naturally quasi-isomorphic to 
the sheaf $\RR_\Ee i_* \underline{R}_{X-\Sigma},$ where $i:X-\Sigma\to X$ denotes the inclusion.
Therefore its cohomology is just the cohomology of $X-\Sigma$, the regular part of $X$:
\[H^i(X-\Sigma,R)\cong \HH^{i}(\mathcal{IC}_{\overline{\infty}}(X))\cong H^{i}(
\RR_\Ee\Gamma(X,\mathcal{IC}_{\overline{\infty}}(X)).\]

We next show that the above sheaf of perverse $\Ee$-algebras is uniquely determined up to quasi-isomorphism by a set of axioms.

Given a sheaf $\Aa$ on $X$ and a subset $i:U\subseteq X$ we will denote by $\Aa|U:=i^*\Aa$
the restriction of $\Aa$ to $U$. By Remark \ref{rmkadjmor} this construction is well-defined in the category of
sheaves of (perverse) $\Ee$-algebras.
For a sheaf $\Aa\in \Sh_X(\widehat{\Pp}\Einf_R)$, its hypercohomology $\HH^*(\Aa)$ is also a sheaf of perverse algebras and we will denote 
$\HH^*(\Aa)_{\ov p}$ the sheaf at perversity $\ov p$.

\begin{definition}\label{axioms}
 A sheaf $\Aa\in \Sh_X(\widehat{\Pp}\Einf_R)$ is said to \textit{satisfy the set of axioms $[\mathrm{AX}]$} if for each perversity $\ov p$:
 \begin{enumerate}
  \item [$(\mathrm{AX}_0)$] The restriction $\Aa|X-\Sigma$ of $\Aa$ to ${X-\Sigma}$ is quasi-isomorphic as a sheaf of perverse $\Ee$-algebras to the sheaf $\mathbb{P}_{\ov \bullet}$ given by the constant sheaf
  $\mathbb{P}_{\ov p}:=\underline{R}_{X-\Sigma}$.
  \item  [$(\mathrm{AX}_1)$] $\HH^i(\Aa)_{\ov p}=0$ for all for all $i<0$.
  \item  [$(\mathrm{AX}_2)$]
  $\HH^i(\Aa|U_{k+1})_{\ov p}=0$ for all $i>p(k)$, $k\geq 2$.
  \item  [$(\mathrm{AX}_3)$] The $\ov p$-\textit{attaching maps} 
  \[\HH^i(j_k^{*}\Aa|U_{k+1})_{\ov p}\lra \HH^i(j_k^*\circ\RR_\Ee i_{k*}\circ i_k^*\Aa|U_{k+1})_{\ov p}\]
  obtained by restricting the natural morphisms
  \[\Aa|U_{k+1}\to \RR_\Ee i_{k*}\circ i_k^*\Aa|U_{k+1},\]
  are isomorphisms for all $i\leq p(k)$ and $k\geq 2$.
 \end{enumerate}
\end{definition}

\begin{lemma}\label{satisfies}
The intersection $\Ee$-algebra $\mathcal{IC}_{\overline{\bullet}}(X,R)$ of Definition \ref{defi_ICE}
satisfies $[\mathrm{AX}]$.
\end{lemma}
\begin{proof}
The restriction of $\mathcal{IC}_{\overline{\bullet}}(X,R)$ to $X-\Sigma$ 
is by definition the sheaf 
\[(i_n\circ \cdots \circ i_3\circ i_2)^*\circ  \tau_{\leq {\bullet}(n)}\circ \RR_\Ee i_{n*}\circ 
\cdots\circ 
\tau_{\leq {\bullet}(3)}\circ \RR_\Ee i_{3*}\circ \tau_{\leq {\bullet}(2)}\circ \RR_\Ee i_{2*}\mathbb{P}_{\ov\bullet}. \]
Since $\tau_{\leq \bullet(k)}$ commutes with inverse images and 
\[\tau_{\leq \bullet(k)}\circ \tau_{\leq \bullet(k')}=\tau_{\leq \bullet(\mathrm{min}(k,k'))},
\]
using the adjunction zig-zag 
\[i_k^* \circ \RR_\Ee i_{k*}\leftarrow \bullet\lra 1\]
we obtain a zig-zag of morphisms of perverse $\Ee$-algebras 
\[\mathcal{IC}_{\overline{\bullet}}(X,R)|X-\Sigma 
\longleftarrow\bullet\lra\tau_{\leq \bullet(2)}\mathbb{P}_{\ov\bullet}=\mathbb{P}_{\ov\bullet},\]
where in the last identity we used the fact $\mathbb{P}_{\ov\bullet}$ is concentrated in degree 0.
The fact that this map is a quasi-isomorphism and the verification of the remaining axioms, which are independent on the multiplicative structure, follows from the additive case 
(see \cite{GMP2}).
In fact, note that the truncation functors are especially designed so 
that the stalk vanishing condition $(\mathrm{AX}_2)$ holds.
\end{proof}

\begin{theorem}\label{uniquenessIC}
If a sheaf $\Aa\in \Sh_X(\widehat{\Pp}\Einf_R)$ satisfies the set of axioms $[\mathrm{AX}]$, then $\Aa$ is naturally quasi-isomorphic as a sheaf of perverse $\Ee$-algebras 
to the intersection $\Ee$-algebra  $\Ss:=\mathcal{IC}_{\overline{\bullet}}(X,R)$.
\end{theorem}
\begin{proof}
We adapt the proof in the additive setting (see for example \cite[Theorem 4.2.1]{BanaglTop}).
By $(\mathrm{AX}_0)$, on $U_2=X-\Sigma$ we have
that $\Aa|U_2$ is quasi-isomorphic to $\Ss|U_2$. Assume inductively that we have constructed a string of quasi-isomorphisms 
from $\Aa|U_k$ to $\Ss|U_k$. 
Recall from Remark \ref{rmkadjmor} that the adjunction morphism
\[\Aa|U_{k+1}\lra \RR_\Ee i_{k*}\circ i^*_k(\Aa|U_{k+1})\]
is a morphism of sheaves of perverse $\Ee$-algebras. This induces a morphism
on truncations 
\[\tau_{\leq  \bullet(k)}\Aa|U_{k+1}\stackrel{\varphi}{\lra}
\tau_{\leq  \bullet(k)}\circ \RR_\Ee i_{k*}\circ i^*_k(\Aa|U_{k+1})\]
Also, we have an inclusion of sheaves of perverse $\Ee$-algebras
\[\psi: \tau_{\leq  \bullet(k)}\Aa|U_{k+1}\hookrightarrow \Aa|U_{k+1}\]
which by condition $(\mathrm{AX}_2)$ is a quasi-isomorphism.
Composing we get maps
\[\Aa|U_{k+1}\stackrel{\psi}{\longleftarrow}\tau_{\leq  \bullet(k)}\Aa|U_{k+1}
\stackrel{\varphi}{\lra}
\tau_{\leq \bullet(k)}\circ \RR_\Ee i_{k*}\circ i^*_k(\Aa|U_{k+1})\cong 
\tau_{\leq  \bullet(k)}\circ \RR_\Ee i_{k*}(\Aa|U_{k}).
\]
By induction hypothesis, we have a string of quasi-isomorphisms
\[\tau_{\leq \bullet(k)}\circ \RR_\Ee i_{k*}(\Aa|U_k)\stackrel{\xi}{\longleftrightarrow} 
\tau_{\leq \bullet(k)}\circ \RR_\Ee i_{k*}(\Ss|U_k)=\Ss|U_{k+1}.
\]
This defines a string of morphisms of perverse $\Ee$-algebras from $\Aa|U_{k+1}$ to $\Ss|U_{k+1}$ extending 
the string from $\Aa|U_{k}$ to $\Ss|U_{k}$. It only remains to show that $\varphi$ is a quasi-isomorphism over 
$U_{k+1}-U_k$. But this follows from $(\mathrm{AX}_3)$. 
\end{proof}

\subsection{Applications}
We compare the different constructions related to the
presence of multiplicative structures in intersection cohomology that
already existed in the literature. We show that our intersection complex of perverse $\Ee$-algebras recovers all of them.

\subsubsection*{Steenrod squares in intersection cohomology}
In \cite{Goresky}, Goresky showed that Steenrod's construction of $\cup_i$-products induces well-defined operations
\[Sq^s:IH_{\ov p}^*(X,\FF_2)\lra IH_{2\ov p}^{*+s}(X,\FF_2)\] on the intersection
cohomology of every topological pseudomanifold $X$, for any perversity $\ov p$
satisfying the condition $2 p(k)\leq k-2$ for all $k$.
These operations satisfy a uniqueness property which amounts to saying that the operations are determined by their value on the regular part of the pseudomanifold (see \cite[Lemma 3.6]{Goresky}). Note that, via axiom $(\mathrm{AX}_0)$, Theorem \ref{uniquenessIC} also states that the $\Ee$-algebra structure is uniquely determined by the algebra structure defined on the regular part of the space. This gives:
\begin{corollary}
The Steenrod operations introduced by Goresky in \cite{Goresky} coincide with the Steenrod operations
determined by the perverse $\Ee$-algebra structure of $\mathcal{IC}_{\overline{\bullet}}(X,R)$ of Definition \ref{defi_ICE}.
\end{corollary}

\begin{remark}
Note that in our setting, the Steenrod operations on intersection cohomology
are also well-defined outside Goresky's bounds
($2 p(k)\leq k-2$) since we added the $\ov\infty$-perversity.
Therefore we obtain operations 
 \[Sq^s:IH_{\ov p}^m(X,\FF_2)\lra IH_{2 \ov p}^{m+s}(X,\FF_2)\]
 for any perversity $\ov p$.
\end{remark}

\subsubsection*{Blown-up cochains}
In \cite{CST}, \cite{CSTGor}, \cite{CSTBl}, the authors introduced and studied a functor of singular cochains $\widetilde{N}^*_{\ov p}$,
the \textit{functor of blown-up cochains}, 
suitable to study products and cohomology
operations of intersection cohomology of stratified spaces.
For any topological pseudomanifold $X$ and any perversity $\ov p\in\Pp$
there is a complex $\widetilde{N}^*_{\ov p}(X,R)$ equipped with a cup product 
whose construction is related to the classical cup product on the singular cochain complex.

When $R=\FF_2$, the complex $\widetilde{N}^*_{\ov p}(X,\FF_2)$ carries an action
of $\Ee(2)$, making it into a perverse $\Ee(2)$-algebra. This gives $\cup_i$-products and hence Steenrod operations,
which are shown to
coincide with the operations introduced by Goresky.
Furthermore, the original bounds are improved, giving 
operations
\[Sq^s:IH_{\ov p}^m(X,\FF_2)\lra IH_{\ov\Ll(\ov p,s)}^{m+s}(X,\FF_2),\]
 where $\ov \Ll(\ov p,s)$ is the perversity given by 
 \[\Ll(\ov p,s)(k):=\mathrm{min}\{2 p(k),p(k)+s\}\]
 and $\ov p$ is any perversity.
 
By forgetting structure, we obtain a functor from perverse $\Ee$-algebras to perverse $\Ee(2)$-algebras, and
there is a sheafification $\widetilde{\mathcal{N}}^*_{\ov \bullet}(X,\ZZ)$
which satisfies the corresponding version for $\Ee(2)$-algebras 
of the set of axioms $[\mathrm{AX}]$ (see \cite{CSTSheaf}). In fact, 
one may naturally extend the $\Ee(2)$-algebra structure on the complex
$\widetilde{N}^*_{\ov p}(X,\FF_2)$ to an $\Ee$-algebra structure in such a way that its sheafification satisfies the set of axioms $[\mathrm{AX}]$.
These facts automatically give:

\begin{theorem}\label{blowncochainsteo}
The sheaves $\widetilde{\mathcal{N}}^*_{\ov \bullet}(X,\ZZ)$ and $\mathcal{IC}_{\overline{\bullet}}(X,\ZZ)$
are naturally quasi-isomorphic as sheaves of perverse $\Ee$-algebras over $\ZZ$.
\end{theorem}

\subsubsection*{Commutative algebras over a field of characteristic zero}Fix $\kk$ a field of characteristic zero.
As explained in Remark \ref{remTW}, when working over $\kk$ we can use Navarro-Aznar's Thom-Whitney 
simple instead of the normalization functor of $\Ee$-algebras, to obtain derived functors for sheaves of commutative dg-algebras.
As a consequence, the results of this section are equally valid for commutative dg-algebras over $\kk$, by replacing 
the operad $\Ee$ with the commutative operad and letting $R=\kk$ everywhere. 

We obtain
a sheaf $\mathcal{IC}_{\overline{\bullet}}(X,\kk)$ of perverse commutative dg-algebras over $\kk$ 
given in each perversity by
\[\mathcal{IC}_{\overline{p}}(X,\kk)= 
\tau_{\leq {p}(n)}\circ \RR_{\TW} i_{n*}\circ 
\cdots\circ 
\tau_{\leq {p}(3)}\circ \RR_{\TW} i_{3*}\circ \tau_{\leq {p}(2)}\circ \RR_{\TW} i_{2*}\underline{\kk}_{U_2}.\]
\begin{definition}
A sheaf $\Aa$ of perverse commutative dg-algebras over $X$ is said to \textit{satisfy the set of axioms $[\mathrm{AX}_\kk]$} if
the conditions of Definition \ref{axioms} are satisfied after replacing $\Ee$ by $\Cc om$ and setting $R=\kk$ everywhere.
\end{definition}
 
The same proof of Lemma \ref{satisfies} shows that $\mathcal{IC}_{\overline{\bullet}}(X,\kk)$ as defined above,
satisfies the set of axioms $[\mathrm{AX}_\kk]$.
Moreover, we also have a uniqueness theorem:

\begin{theorem}\label{uniquenessICQ}
If a sheaf $\Aa$ of perverse commutative dg-algebras over $X$ satisfies the set of axioms $[\mathrm{AX}_\kk]$, then $\Aa$ is naturally quasi-isomorphic as a sheaf of perverse commutative dg-algebras
to the intersection dg-algebra $\mathcal{IC}_{\overline{\bullet}}(X,\kk)$.
\end{theorem}

When forgetting the multiplicative structures, the sheaf $\mathcal{IC}_{\overline{\bullet}}(X,\kk)$ is naturally quasi-isomorphic to Deligne's intersection complex. In particular, our construction over $\QQ$ is a sheaf-theoretic solution to the problem of commutative cochains in the intersection setting.

The main construction in \cite{CST} is a perverse dg-algebra $IA_{\ov\bullet}(X)$ defined over $\QQ$, for any given topological pseudomanifold. This perverse algebra plays the role of Sullivan's 
algebra of piece-wise linear forms, and it may be interpreted as an extension of
Sullivan's presentation of rational homotopy type to intersection cohomology.
The perverse algebra $IA_{\ov\bullet}(X)$ is defined via singular cochains, using simplicial blow-ups.
We next compare this perverse dg-algebra with our sheaf-theoretic approach.
Denote by
\[S_{\ov\bullet}(X,R):=\RR_{\TW}\Gamma(X,\mathcal{IC}_{\overline{\bullet}}(X,\QQ))\]
the perverse
 commutative dg-algebra  given by taking derived global sections of
the intersection commutative dg-algebra of $X$. We have:
\begin{theorem}\label{IAteo}
The objects $IA_{\ov\bullet}(X)$ and $S_{\ov\bullet}(X,\QQ)$ are quasi-isomorphic as perverse
 commutative dg-algebras.
\end{theorem}
\begin{proof}
We consider the following three functors from the simplex category $\Delta$ to the category of rational differential graded algebras: the normalized cochain complex $N$,
Sullivan's functor of piece-wise linear forms
\[A_{PL}([n]):={{\Lambda(t_0,\cdots,t_\alpha,dt_0,\cdots,dt_\alpha)}\over{\sum t_i-1,\sum dt_i}}\]
and the functor $T:=N\otimes A_{PL}$.
These functors define extendable universal coefficient systems in the sense of 
\cite[Definition 1.25]{CST} and by Theorem A of loc. cit., they extend to functors
$\widetilde N_{\ov p}$, $\widetilde A_{PL,\ov p}$, and $\widetilde T_{\ov p}$
from the category of filtered face sets to
commutative dg-algebras.
For our purposes, it suffices to note that there is a functor $S_f$ from topological pseudomanifolds to filtered face sets
analogous to the singular simplicial functor sending a topological space to a simplicial set, and that by definition we have 
\[IA_{\ov p}(X):=\widetilde A_{PL,\ov p}(S_{f}(X)).\]
Moreover, for any filtered face set $K$, the map
$\widetilde A_{PL,\ov p}(K)\to \widetilde N_{\ov p}(K)$
induced by the integration map is a quasi-isomorphism (see Corollary 1.39 of \cite{CST}) In particular, we obtain quasi-isomorphisms
\[\widetilde A_{PL,\ov p}(K)\stackrel{\sim}{\lra} \widetilde T_{\ov p}(K)\stackrel{\sim}{\longleftarrow} \widetilde N_{\ov p}(K).\]
After sheafification, we still have quasi-isomorphisms 
\[\widetilde \Aa_{PL,\ov p}\stackrel{\sim}{\lra} \widetilde{\mathcal{T}}_{\ov p}\stackrel{\sim}{\longleftarrow} \widetilde \Nn_{\ov p}.\]
Since $\Nn_{\ov \bullet}$ satisfies $[\mathrm{AX}]$, it follows that $\widetilde \Aa_{PL,\ov \bullet}$ also satisfies $[\mathrm{AX}]$,
which in turn implies that $IA_{\ov \bullet}(X)$ satisfies $[\mathrm{AX}_\QQ]$.
The result now follows from Theorem \ref{uniquenessICQ}.
\end{proof}

Over the field of real numbers, there is also a sheaf $\Ii\Omega^*_{\ov\bullet}(X)$
of \textit{intersection differential forms},
which satisfies the set of axioms $[\mathrm{AX}_{\RR}]$
(see \cite{BHS}, \cite{Brylinski}, \cite{Pollini},
\cite{Martinxo}). As a consequence of Theorem \ref{uniquenessICQ}
we have:

\begin{corollary}\label{IdR}
The sheaves $\Ii\Omega^*_{\ov\bullet}(X)$ and $\mathcal{IC}_{\overline{\bullet}}(X,\RR)$
are naturally quasi-isomorphic as sheaves of perverse
 commutative dg-algebras.
\end{corollary}

\begin{remark}
Theorem \ref{blowncochainsteo} and Corollary \ref{blowncochainsteo} together imply that the multiplicative structures on real blown-up cochains
$\widetilde{\mathcal{N}}^*_{\ov \bullet}(X,\RR)$ and on intersection differential forms $\Ii\Omega^*_{\ov\bullet}(X)$ agree, which was previously not known.
\end{remark}

\bibliographystyle{amsplain}
\bibliography{bibliografia}

\end{document}